\documentclass[12pt,leqno]{amsart}
\usepackage{amsfonts,amsmath,amsthm,amssymb,amscd}
\usepackage{pstricks,pst-node}
\usepackage{graphicx}
\usepackage{epstopdf}
\usepackage[vcentermath]{youngtab}
\usepackage[all]{xy}
\usepackage{color}

\newtheorem{theorem}{Theorem}

\newtheorem{lemma}[theorem]{Lemma}
\newtheorem{corollary}[theorem]{Corollary}
\newtheorem{conjecture}[theorem]{Conjecture}
\theoremstyle{definition}
\newtheorem{problem}[theorem]{Problem}
\newtheorem{remark}[theorem]{Remark}
\newtheorem{example}[theorem]{Example}

\numberwithin{equation}{section}
\numberwithin{theorem}{section}

\newcounter{thmenum}
\newenvironment{thmenumerate}{%
\begin{list}{$(\thethmenum)$}{%
\usecounter{thmenum}
\setlength{\labelsep}{.5em}
\setlength{\labelwidth}{-7pt}
\setlength{\topsep}{0pt}
\setlength{\partopsep}{0pt}
\setlength{\parsep}{0pt}
\setlength{\leftmargin}{3pt}
\setlength{\rightmargin}{0pt}
\setlength{\itemindent}{\leftmargin}
\setlength{\itemsep}{0pt}
}}
{\end{list}}

\newcommand\CC{\mathbb{C}}
\newcommand\RR{\mathbb{R}}
\newcommand\ZZ{\mathbb{Z}}

\newcommand\caln{\mathcal{N}}

\newcommand\frakg{\mathfrak{g}}
\newcommand\frakh{\mathfrak{h}}
\newcommand\frakk{\mathfrak{k}}
\newcommand\frakl{\mathfrak{l}^{\mathfrak{q}}}
\newcommand\frako{\mathfrak{o}}

\newcommand\frakq{\mathfrak{q}}
\newcommand\fraks{\mathfrak{s}}

\newcommand\fraksp{\mathfrak{sp}}
\newcommand\fraku{\mathfrak{u}}
\newcommand\frakgl{\mathfrak{gl}}

\newcommand{\GR}{G_{\RR}}

\newcommand{\LR}{L_{\RR}}
\newcommand{\LqR}{L_{\RR}^{\mathfrak{q}}}
\newcommand{\GpR}{G_{\RR}'}

\newcommand{\psgR}{P_{\RR}}

\newcommand{\calorbit}{\mathcal{O}}
\newcommand{\bborbit}{\mathbb{O}}

\newcommand\Gorbit{\calorbit^G}
\newcommand\Korbit{\bborbit^K}
\newcommand\Gporbit{\calorbit^{G'}}
\newcommand\Gpporbit{\calorbit^{G''}}
\newcommand\Kporbit{\bborbit^{K'}}

\newcommand\Korbitp{{\bborbit'}^{K}}

\newcommand{\Ograph}{\Gamma}
\DeclareMathOperator{\AVgraph}{\mathcal{AV}^\Gamma}

\DeclareMathOperator{\syd}{SYD}
\DeclareMathOperator{\yd}{YD}
\newcommand\sydbdi{\syd_{\mathrm{BDI}}}
\newcommand\sydci{\syd_{\mathrm{CI}}}
\newcommand\sydcii{\syd_{\mathrm{CII}}}
\newcommand\syddiii{\syd_{\mathrm{DIII}}}
\newcommand\sydx{\syd_{\mathrm{X}}}
\newcommand\ydbdi{\yd_{\mathrm{BDI}}}
\newcommand\ydci{\yd_{\mathrm{CI}}}
\newcommand\ydcii{\yd_{\mathrm{CII}}}
\newcommand\yddiii{\yd_{\mathrm{DIII}}}
\DeclareMathOperator{\Lie}{Lie}
\DeclareMathOperator{\Aut}{Aut}
\DeclareMathOperator{\Ind}{Ind}

\DeclareMathOperator{\AV}{\mathcal{AV}}
\DeclareMathOperator{\Mat}{Mat}

\DeclareMathOperator{\ind}{ind}
\DeclareMathOperator{\gind}{g\text{-}ind}

\newcommand{\restrict}{\big|}

\newcommand{\closure}[1]{\overline{#1}}
\newcommand{\conjugate}[1]{\overline{#1}}
\newcommand{\lie}[1]{\mathfrak{#1}}
\newcommand{\transpose}[1]{{}^t{#1}}

\makeatletter
\def\abDab{\@ifnextchar[\abDab@{ab\cdots ab}}
\def\abDab@[#1]{\underbrace{ab\cdots ab}_{#1}}
\def\abDba{\@ifnextchar[\abDba@{ab\cdots ba}}
\def\abDba@[#1]{\underbrace{ab\cdots ba}_{#1}}
\def\baDab{\@ifnextchar[\baDab@{ba\cdots ab}}
\def\baDab@[#1]{\underbrace{ba\cdots ab}_{#1}}
\def\baDba{\@ifnextchar[\baDba@{ba\cdots ba}}
\def\baDba@[#1]{\underbrace{ba\cdots ba}_{#1}}
\def\abDDab{\@ifnextchar[\abDDab@{ab\cdots\cdots ab}}
\def\abDDab@[#1]{\overbrace{ab\cdots\cdots ab}^{#1}}
\def\abDDba{\@ifnextchar[\abDDba@{ab\cdots\cdots ba}}
\def\abDDba@[#1]{\overbrace{ab\cdots\cdots ba}^{#1}}
\def\baDDab{\@ifnextchar[\baDDab@{ba\cdots\cdots ab}}
\def\baDDab@[#1]{\overbrace{ba\cdots\cdots ab}^{#1}}
\def\baDDba{\@ifnextchar[\baDDba@{ba\cdots\cdots ba}}
\def\baDDba@[#1]{\overbrace{ba\cdots\cdots ba}^{#1}}
\makeatother

\newcommand\Rorbit{\bborbit^{G_\RR}}
\newcommand\Rporbit{\bborbit^{G'_\RR}}
\newcommand\call{\mathcal{L}}
\DeclareMathOperator{\AS}{\mathrm{AS}}
\DeclareMathOperator{\WF}{\mathrm{WF}}

\newcommand\frakn{\mathfrak{n}}
\usepackage{verbatim}

\newcounter{rbenum}
\newenvironment{rbenumerate}{%
\begin{list}{{\upshape(\therbenum)}}{%
\usecounter{rbenum}
\setlength{\labelsep}{.5em}
\setlength{\labelwidth}{-7pt}
\setlength{\topsep}{0pt}
\setlength{\partopsep}{0pt}
\setlength{\parsep}{0pt}
\setlength{\leftmargin}{3pt}
\setlength{\rightmargin}{0pt}
\setlength{\itemindent}{\leftmargin}
\setlength{\itemsep}{0pt}
}}
{\end{list}}

\newcommand{\mbfm}{\boldsymbol{m}}
\newcommand{\mbfp}{\boldsymbol{p}}

\newcommand{\fboxplus}{\Yvcentermath1\Yboxdim11pt\young(+)}
\newcommand{\fboxminus}{\Yvcentermath1\Yboxdim11pt\young(-)}

\newcommand{\version}{Ver.~0.0}
\newcommand{\setversion}[1]{\renewcommand{\version}{Ver.~{#1}}}
\setversion{0.1 [2010/07/27 23:26]}
\setversion{0.2 [2010/09/09 16:59]}
\setversion{0.3 [2011/01/31 15:33]}
\setversion{0.4 [2011/12/26 16:21]}
\setversion{0.5 [2011/12/29 15:47]}
\setversion{0.9 [2012/03/12 16:18]}
\setversion{1.0 [2012/04/02 13:14]}
\setversion{1.1 [2014/03/03 17:56:32]}
\setversion{1.5 [2014/08/30]}
\setversion{1.7 [2014/09/12 21:57:31]}
\setversion{1.71 [2014/09/16 23:15:09]}
\setversion{1.72 [2014/10/06 17:46:45]}

\title [Orbit graph of associated varieties]
{Codimension one connectedness of the graph of associated varieties}

\author{Kyo Nishiyama}
\address{
Department of Physics and Mathematics\\
Aoyama Gakuin University\\
Fuchinobe 5-10-1, Sagamihara 252-5258, Japan}
\email{kyo@gem.aoyama.ac.jp}

\thanks{Supported by JSPS Grant-in-Aid for Scientific Research (B) \#{21340006}.}

\author{Peter Trapa}
\address{Department of Mathematics, University of Utah, Salt Lake City, UT
84112, USA}
\email{ptrapa@math.utah.edu}

\thanks{}

\author{Akihito Wachi}
\address{
Department of Mathematics \\
Hokkaido University of Education \\
Shiroyama 1, Kushiro 085-8580 \\
Japan
}
\email{wachi@kus.hokkyodai.ac.jp}

\thanks{Supported by JSPS Grant-in-Aid for Scientific Research (C) \#{23540179}.}

\date{\version\quad(compiled on \today)}

\begin{document}

\maketitle

\begin{abstract}
\quad\\
Let $ \pi $ be an irreducible Harish-Chandra $ (\lie{g}, K) $-module, and 
denote its associated variety by $ \AV(\pi) $.  
If $ \AV(\pi) $ is reducible, then each irreducible component must contain 
codimension one boundary component.  
Thus we are interested in the codimension one adjacency of nilpotent orbits for a symmetric pair $ (G, K) $.  
We define the notion of orbit graph and associated graph for $ \pi $, and study its structure for classical 
symmetric pairs; 
number of vertices, edges, connected components, etc.  
As a result, we prove that the orbit graph is connected for even nilpotent orbits.  

Finally, for indefinite unitary group $ U(p, q) $, we prove that 
for each connected component of the orbit graph $ \Ograph_K(\Gorbit_{\lambda}) $ thus defined, 
there is an irreducible Harish-Chandra module $ \pi $ whose associated graph is exactly equal to the connected component.
\end{abstract}

\tableofcontents

\section{Introduction}

Let $G$ be a connected reductive complex algebraic group,
and $(G,K)$ a symmetric pair,
that is,
$K$ is the fixed point subgroup of a non-trivial involution $\theta \in \Aut(G)$.  
Note that $ K $ need not be connected.  
The differential of the involution $\theta$ gives an automorphism of order two
of $\frakg = \Lie(G)$, which we will denote by the same letter.  
Let $\frakk$ and $\fraks$ be the eigenspaces of $ \theta $ with 
the eigenvalues $+1$ and $-1$, respectively.
Then a direct sum $\frakg = \frakk + \fraks$ gives  
the (complexified) Cartan decomposition
corresponding to the symmetric pair $(G,K)$.

Let $\caln(\fraks)$ be the set of nilpotent elements in $\fraks$,
which is a closed subvariety of $\fraks$,
and called the nilpotent variety of $\fraks$.
We call $K$-orbits in $\caln(\fraks)$
\emph{nilpotent orbits for a symmetric pair}.
%

It follows from Kostant-Rallis \cite{Kostant.Rallis.1971}
that the number of the $K$-orbits in $\caln(\fraks)$ is finite.
Moreover, the classification of nilpotent $ K $-orbits is completely known 
for simple $ G $, and if $G$ is classical, 
it is 
given combinatorially 
in terms of signed Young diagrams (see, e.g., \cite{Collingwood.McGovern.1993}).

When two nilpotent $K$-orbits in $\caln(\fraks)$ generate the same $G$-orbit $\Gorbit$ in $\frakg$,
we call these two $K$-orbits are 
\emph{adjacent in codimension one} (or simply \emph{adjacent}) 
if the intersection of their closures contains a $ K $-orbit 
of codimension one. 
We consider a non-oriented {graph} $\Ograph_K(\Gorbit)$
with the vertices consisting of $K$-orbits on $\caln(\fraks)$ contained in $\Gorbit$,
and edges drawn if two $K$-orbits are adjacent.  
The graph is called an \emph{orbit graph}.
We study combinatorial structures of the graph $\Ograph_K(\Gorbit)$,
which are related to representation-theoretic problem
on the geometry of associated varieties of Harish-Chandra modules.

For example, the number of vertices of $\Ograph_K(\Gorbit)$ gives the number of nilpotent $ K $-orbits which 
generates the same $ \Gorbit $.  
This roughly classifies irreducible Harish-Chandra modules 
with a fixed infinitesimal character 
which have annihilators with the same associated variety.  
We give {generating functions} 
of the number of the nilpotent orbits 
for classical symmetric pairs in \S 
\ref{section:generating.function.of.K.orbits}.  
There we also give generating functions of 
the number of vertices of $\Ograph_K(\Gorbit)$ for individual orbits.

From a viewpoint of representation theory, 
nilpotent $K$-orbits in $\caln(\fraks)$ and their closures 
occur as  irreducible components of the associated varieties of Harish-Chandra modules.
For an irreducible Harish-Chandra module $X$, 
its associated variety $\AV(X)$ decomposes into irreducible components as 
\begin{align}
  \label{eqn:AV.decompose.into.K.orbits}
  \AV(X) = \bigcup_{i = 1}^{\ell} \closure{\Korbit_i} ,
\end{align}
where $\Korbit_i$ are nilpotent $K$-orbits in $\caln(\fraks)$, 
which generate a common nilpotent $ G $-orbit $ \Gorbit $.  
The closure of the $ G $-orbit $ \Gorbit $ is an associated variety of the 
primitive ideal of $ X $.
Thus we get a full subgraph of $\Ograph_K(\Gorbit)$ with vertices 
\begin{equation*}
\{ {\Korbit_i} \mid \closure{\Korbit_i} \text{ is an irreducible component of $ \AV(X) $} \} .  
\end{equation*}
We denote this subgraph by $ \AVgraph(X) $, and call it an \emph{associated graph} of $ X $.  
Here we omit the subscript $ K $, because the Harish-Chandra module $ X $ already encodes it.

Vogan's theorem (\cite[Theorem~4.6]{Vogan.1991}) suggests that the following conjecture is plausible to hold.

\begin{conjecture}
If $ X $ is an irreducible Harish-Chandra $ (\lie{g}, K) $-module, 
the associated graph $ \AVgraph(X) $ is connected.
\end{conjecture}

In the case of a symmetric pair of type AIII, 
we will prove 

\begin{theorem}[Theorem \ref{thm:for-AIII} below]
\label{intro.theorem:main.result.AIII}
Let $ G_{\RR} = U(p,q) $, an indefinite unitary group, 
and $ (G, K) = (GL_{n}(\CC), GL_{p}(\CC) \times GL_{q}(\CC)) \; (n = p + q) $ be an associated symmetric pair of type {\upshape{AIII}}.  
Let us consider a nilpotent $ G $-orbit $ \Gorbit $ in $ \lie{g} $.  
For any connected component in the orbit graph $\Ograph_K(\Gorbit)$, 
there exists an irreducible Harish-Chandra $ (\lie{g}, K) $-module $X$ 
whose associated graph $\AVgraph(X)$ 
is exactly the chosen connected component.  
\end{theorem}

This theorem is a partial converse to the conjecture above.  
For a general classical symmetric pair including type AIII, 
we also have the following

\begin{theorem}
\label{intro.theorem:even.nilpotent.orbit.and.deg.PS}
Let $ (G, K) $ be a classical symmetric pair corresponding to 
a real form $ G_{\RR} $ of $ G $.  
If $ \Gorbit $ is an even nilpotent orbit, then the orbit graph $ \Ograph_K(\Gorbit) $ is connected, and 
there exists an irreducible degenerate principal series representation $ \pi $ of $ G_{\RR} $ 
such that $ \AVgraph(\pi) = \Ograph_K(\Gorbit) $.
\end{theorem}

For this, see Remark~\ref{remark:even.nilpotent.is.associated.variety}.

These theorems show that 
the combinatorial structure of orbit graphs seems important and interesting.  
In \S~\ref{section:structure.of.orbit.graph.AIII}, for a symmetric pair of type AIII, 
we study the structure of the orbit graph $ \Ograph_K(\Gorbit) $, 
and obtain a combinatorial description of $ \Ograph_K(\Gorbit) $ in 
Theorem \ref{thm:AIII-description-of-graph}.
In particular, we can give an explicit formula which gives the number of connected components of the graph.  
For the precise statement, see Theorem \ref{thm:number.of.connected.components.of.orbit.graph} 
and the arguments before it.

The main tool of our arguments is an induction of graphs introduced in \S~\ref{subsec:AIII-induction.of.subgraph}.  
The induction carries a connected component of the orbit graph of a smaller nilpotent orbit to that of a larger (or induced) nilpotent orbit.  

The combinatorial arguments in \S~\ref{section:structure.of.orbit.graph.AIII} can be carried over to the other classical symmetric pairs.  
The results thus obtained are summarized in \S~\ref{section:other.classical.symmetric.pairs}; 
among them, we determine the connected components of orbit graphs and prove that there is only one connected component for an even nilpotent orbit 
(a part of the claim of Theorem~\ref{intro.theorem:even.nilpotent.orbit.and.deg.PS}).  

Theorem~\ref{intro.theorem:main.result.AIII} above is proved in \S~\ref{section:AV.of.HC.module} for type AIII.  
Essentially this theorem claims that the induction of orbit graphs described in purely combinatorial manner and 
the cohomological (or parabolical) induction of representations match up.  
It is natural to expect a similar result for other symmetric pairs 
and our combinatorial arguments in \S~\ref{section:other.classical.symmetric.pairs} strongly suggest such statements.  
This is a future subject of ours.


\section{Preliminaries}

Let $ G $ be a connected reductive algebraic group over the complex number field $ \CC $.  
Let $ G_\RR $ be the connected component of the identity of a noncompact real form of $ G $.  
We denote by $ K_\RR $ a maximal compact subgroup of $ G_\RR $, so that 
$(G_\RR, K_\RR)$ is a symmetric pair with respect to a Cartan involution.  
Let $\frakg_\RR$ and $\frakk_\RR$ be the Lie algebras of $ G_\RR $ and $ K_\RR $ respectively, and 
$\frakg_\RR =  \frakk_\RR + \fraks_\RR$ be the associated Cartan decomposition.  
In general, we denote by $ H_\RR $ a real Lie group, and $ H $ its complexified algebraic group (if it exists).    
We also use corresponding German small letters to denote their Lie algebras; 
so $ \frakh_\RR $ is the Lie algebra of $ H_\RR $ and $ \frakh $ its complexification.

Pick a nilpotent $ G $-orbit $ \Gorbit $ in $ \frakg $, and 
let 
\begin{equation}
\label{eqn:G.orbit.decompose.into.K.orbits}
\Gorbit \cap \fraks = \coprod_{k=1}^m \Korbit_k
\end{equation}
be the decomposition of $\Gorbit$ into equidimensional Lagrangian $K$-orbits 
(see, e.g., \cite[Corollary~5.20]{Vogan.1991}).  
We will denote a nilpotent $G$-orbit in $\frakg$ by $\Gorbit$
(or $\Gorbit_\lambda$ when it is parameterized by a partition $\lambda$ in the classical cases),
and a nilpotent $K$-orbit in $\fraks$ by $\Korbit$
(or $\Korbit_T$ when parameterized by a signed Young diagram $T$).

Two nilpotent $K$-orbits $\Korbit_k$ and $\Korbit_{\ell}$ are said to be 
\emph{adjacent} 
if these two nilpotent $ K $-orbits appear in  
the decomposition
\eqref{eqn:G.orbit.decompose.into.K.orbits} of $ \Gorbit $, and they 
share a boundary of codimension one. 
Also we say two nilpotent orbits 
$ \Korbit $ and $ \Korbitp $ are \emph{connected in codimension one} 
if there exists a sequence of nilpotent $ K $-orbits $ \Korbit = \Korbit_{k_1}, \Korbit_{k_2}, \dots , 
\Korbit_{k_r} = \Korbitp $ such that each successive pair $ ( \Korbit_{k_i} , \Korbit_{k_{i+1}} ) $ 
is an adjacent pair.  

We define a graph $\Ograph_K(\Gorbit)$ 
with vertices $ \{ \Korbit_1, \Korbit_2, \ldots, \Korbit_m \} $ 
and edges given by the adjacency relation.  
The graph $ \Ograph_K(\Gorbit) $ is called an \emph{orbit graph}.

Now let $ X $ be an irreducible Harish-Chandra $ (\frakg, K) $-module and 
let 
\begin{equation*}
\AV(X) = \bigcup_{i=1}^{\ell} \closure{\Korbit_i}
\end{equation*}
be the irreducible decomposition
of its associated variety.  
The labeling of $ \Korbit_i $ by $ i = 1, \dots, \ell $ is now different from those 
which are used in \eqref{eqn:G.orbit.decompose.into.K.orbits}, 
but it is known that each $ \Korbit_i $ will generate the same nilpotent $ G $-orbit $ \Gorbit = \Gorbit_X $.  
In fact, $ \closure{\Gorbit_X} $ is the associated variety of the primitive ideal of $ X $.  
Therefore we can consider $ \{ \Korbit_1, \dots, \Korbit_{\ell} \} $ 
as a subset of vertices of $ \Ograph_K(\Gorbit_X) $, and we define the full subgraph 
\begin{math}
  \AVgraph(X) 
\end{math}
of $ \Ograph_K(\Gorbit_X) $, 
whose vertices are the irreducible components of $\AV(X)$,
and whose edges are the ones in $\Ograph_K(\Gorbit_X)$.

Vogan proved in 
\cite[Theorem~4.6]{Vogan.1991} 
that 
the codimension in $\Korbit_i$ of its boundary
$\partial \Korbit_i = \closure{\Korbit_i} \setminus \Korbit_i$
is equal to one 
if $\AV(X)$ is reducible (i.e., $\ell \ge 2$).

The boundary of codimension one of the closure of a nilpotent $ K $-orbit $ \Korbit $ is generally reducible, and 
one of its irreducible components might be contained in the closure of 
another $K$-orbit $\Korbit_0$, hence $ \Korbit $ and $ \Korbit_0 $ are adjacent; 
or it might be only contained in $\Korbit$ itself, 
so it does not contribute to the connectedness in codimension one.  
Both cases are possible and actually occur.
However, it is plausible that the following conjecture holds.
In Conjecture~\ref{conj:codim-one-connected} and
Problem~\ref{prob:general-case} below, $K$ is not necessarily connected.  
In fact, if we take a connected component of the fixed point subgroup of the involution $ \theta $, 
the claim of the conjecture becomes even stronger.  

\begin{conjecture}\quad
\label{conj:codim-one-connected}
Let $X$ be an irreducible Harish-Chandra $ (\mathfrak{g}, K) $-module,
and 
$\AV(X) = \bigcup_{i=1}^{\ell} \overline{\Korbit_i}$
the irreducible decomposition of its associated variety.
Then the graph $\AVgraph(X)$ is connected.
Namely, for any pair $ ( \Korbit_i , \Korbit_j ) $, 
there exist a sequence of nilpotent $K$-orbits 
\begin{equation*}
\Korbit_i = \Korbit_{i_0},  \; \Korbit_{i_1}, \; \Korbit_{i_2}, \; \dots , \;
\Korbit_{i_n} = \Korbit_j 
\end{equation*}
such that 
$ \closure{\Korbit_{i_{k- 1}}} \cap \closure{\Korbit_{i_k}} \; (1 \leq k \leq n) $ contains a 
nilpotent $ K $-orbit of codimension one.
\end{conjecture}

Taking this conjecture into account, in this paper, 
we consider the following problems.
First three are combinatorial problems, 
and remaining two are representation-theoretic ones.  

\begin{problem}
\label{prob:general-case}
Let us consider a symmetric pair $ (G, K) $ as above, and 
let $\Gorbit$ be a nilpotent $G$-orbit in $\frakg$.
\begin{rbenumerate}
\item
\label{prob:item:graph.structure}
Describe the explicit structure of the orbit graph $ \Ograph_K(\Gorbit) $.

\item
\label{prob:item:number.of.connected.components.of.graph}
Find the number of connected components of $ \Ograph_K(\Gorbit) $.

\item
\label{prob:item:number.of.K.orbits}
Find the number of $ K $-orbits in $ \Gorbit \cap \fraks $.

\item
\label{prob:item:irrep.for.connected.G.orbit}
Assume that the graph $\Ograph_K(\Gorbit)$ is connected.  
Does there exist an irreducible Harish-Chandra
$ (\frakg, K) $-module $X$ 
such that $\Ograph_K(\Gorbit) = \AVgraph(X)$?

\item
\label{prob:item:irrep.for.connected.component}
More generally,
for any connected component $Z \subset \Ograph_K(\Gorbit)$,
does there exist an irreducible Harish-Chandra module $X$
such that $Z = \AVgraph(X)$?
Here a connected component of a graph means a maximal connected full subgraph.

\end{rbenumerate}
\end{problem}

We will answer most of these problems in the classical cases.

If the intersection of $ G $-orbits with $ \lie{s} $ is always a single $ K $-orbit, 
most of our problems above become trivial.  So we omit these cases.  
However, our problem does hold in such cases.  

Thus, in the following, we only 
consider classical symmetric pairs 
of type AIII, BDI, CI, CII, DIII 
in the notation of \cite[Chapter X, Table V]{Helgason.1978}.

\section{The number of nilpotent orbits for a symmetric pair}
\label{section:generating.function.of.K.orbits}

In this section, 
we solve 
Problem~\ref{prob:general-case}~\eqref{prob:item:number.of.K.orbits} 
for the classical symmetric pairs.  
For classical symmetric pairs, a classification of $K$-orbits in $\fraks$ 
and their closure relations
are obtained 
by Takuya Ohta \cite{Ohta.1986} 
(see also \cite{Kraft.Procesi.1979}, \cite{Burgoyne.Cushman.1977} and \cite{Djokovic.1982}) 
and we use Ohta's result in the following case-by-case arguments.

\subsection{Type AIII $(GL_{p+q}(\CC), GL_{p}(\CC) \times GL_{q}(\CC))$}
In the following, we denote $GL_n(\CC)$ simply by $GL_n$ and use similar abbreviation for other classical groups.
Let us consider a symmetric pair
\begin{equation*}
(G, K ) = (GL_{n}, GL_{p} \times GL_{q}) \qquad 
(n = p+q), 
\end{equation*}
where $ K $ is embedded into $ G $ block diagonally.  
Thus the corresponding Cartan decomposition is
\begin{equation*}
\lie{g} = \lie{k} \oplus \lie{s} , \qquad
\lie{k} = \lie{gl}_p \oplus \lie{gl}_q, 
\quad
\lie{s} = \Mat(p,q; \CC) \oplus \Mat(q,p; \CC), 
\end{equation*}
where $ \lie{s} $ is anti-diagonally embedded into $ \lie{g} $.  

Let us first recall that the nilpotent $ G = GL_n $-orbits 
in $ \frakg = \lie{gl}_n $ 
are parameterized by the partitions of $ n $, 
i.e., 
collections of the size of Jordan blocks arranged in non-increasing order.  
To each partition $ \lambda = ( \lambda_1, \dots, \lambda_{\ell}) $ of $ n $, 
we associate a nilpotent orbit denoted by $ \Gorbit_{\lambda} $.
When $ \Gorbit_{\lambda} $ is given, 
a connected component of 
its intersection $ \Gorbit_{\lambda} \cap \fraks $ 
with $\fraks $
is a nilpotent $ K $-orbit in $\fraks$, 
and every nilpotent $ K $-orbit in $\fraks$ appears in this way.  
It is known that these $K$-orbits are parameterized by 
the \emph{signed Young diagrams on $\lambda$ of signature $(p,q)$}:
\begin{align*}
  \Gorbit_\lambda \cap \fraks
  =
  \coprod_{T \in \syd(\lambda;p,q)} \Korbit_T.
\end{align*}
Here $\syd(\lambda; p,q)$ denotes 
the set of 
signed Young diagrams $ T $ on $\lambda$ of signature $(p,q)$ which satisfy
\begin{rbenumerate}
\item
$T$ has the same shape as $\lambda$.
\item
There are $p$ boxes with $({+})$-sign and 
$q$ boxes with $({-})$-sign in $ T $.
\item
Signs are alternating in each row
(in columns signs may run in any order).
\end{rbenumerate}
From this description, 
we get the generating function of the number of the nilpotent $K$-orbits
on $\fraks$ as follows.

\begin{theorem}
\label{thm:AIII-generating-funct}
Denote a partition $\lambda$ of $n$ 
as $\lambda = [1^{m_1} \cdot 2^{m_2} \cdots n^{m_n}]$ 
by using the multiplicities $m_i$ of $i$.
Then we have
\begin{multline}
  \label{eq:AIII-generating-funct}
  \sum_{\substack{p, \, q \ge 0,  \; \lambda \vdash (p+q)}}
  \# \syd(\lambda;p,q) \, a^p b^q t_{\lambda}
  \\ 
  =
  \prod_{k=1}^\infty \dfrac{1}{(1-a^k b^k t_{2k})^2} \cdot
  \dfrac{1}{1-a^{k-1} b^k t_{2k-1}} \cdot
  \dfrac{1}{1-a^k b^{k-1} t_{2k-1}},
\end{multline}
where $t_{\lambda} = t_1^{m_1} t_2^{m_2} \cdots$,
and $\lambda \vdash (p+q) $ means $\lambda$ is a partition of $p+q$.
This formula is an equality in the ring of formal power series in variables 
$a$, $b$, $t_1$, $t_2$, $\ldots$.
\end{theorem}
\begin{proof}
Set $\syd = \bigcup_{p,q\ge0, \lambda\vdash p+q} \syd(\lambda;p,q)$,
and define the map $\phi$ by
\begin{align*}
\begin{array}{ccccc}
\phi: & \syd & \to & \CC[\hspace{-0.25ex}[a,b,t_1,t_2,\ldots]\hspace{-0.25ex}]
\\
& T & \mapsto & a^p b^q t_\lambda & T \in (\syd(\lambda;p,q)).
\end{array}
\end{align*}
Then it is obvious that
$\sum_{T\in\syd} \phi(T)$ is equal to the left-hand side of 
(\ref{eq:AIII-generating-funct}).

A signed Young diagram is a union of rows of the following four types:
{
\newcommand{\tinycdots}{\raisebox{1pt}[5pt][0pt]{$\scriptscriptstyle\cdots$}}
\Yboxdim10pt%
\begin{align*}
\epsilon_k^+ &= \scriptsize\young(+-+-\tinycdots\tinycdots+-) 
&& (\text{length is $2k$}), \\
\epsilon_k^- &= 
 \scriptsize\young(-+-+\tinycdots\tinycdots-+)
&& (\text{length is $2k$}), \\
\delta_k^+ &= 
 \scriptsize\young(+-+\tinycdots\tinycdots-+)
&& (\text{length is $2k-1$}), \\
\delta_k^- &= 
\scriptsize\young(-+-\tinycdots\tinycdots+-)
&& (\text{length is $2k-1$}).
\end{align*}
\\
We call these diagrams {\em primitives} of signed Young diagrams of type AIII.
}%
Using primitives we can write $\syd$ as 
\begin{align*}
\syd = \left\{
\sum_{k\ge0} (e_k^+ \epsilon_k^+ + e_k^- \epsilon_k^- +
d_k^+ \delta_k^+ + d_k^- \delta_k^-)
\biggm|
e_k^\pm, d_k^\pm \ge 0
\right\},
\end{align*}
where sum means the sum of rows.
Thus we have 
\begin{align*}
\sum_{T\in\syd} \phi(T) 
& =
\sum_{\substack{ 
    e_1^\pm, e_2^\pm, \ldots \ge 0, \\ d_1^\pm, d_2^\pm, \ldots \ge 0
}}
\phi\left(
\sum_{k\ge1} (e_k^+ \epsilon_k^+ + e_k^- \epsilon_k^- +
d_k^+ \delta_k^+ + d_k^- \delta_k^-)
\right)
\\ &=
\sum_{\substack{ 
    e_1^\pm, e_2^\pm, \ldots \ge 0, \\ d_1^\pm, d_2^\pm, \ldots \ge 0
}}
\prod_{k\ge1} \phi(\epsilon_k^+)^{e_k^+} \phi(\epsilon_k^-)^{e_k^-}
\phi(\delta_k^+)^{d_k^+} \phi(\delta_k^-)^{d_k^-}
\\ &=
\prod_{k\ge1} 
\sum_{e_k^+\ge0} \phi(\epsilon_k^+)^{e_k^+} 
\sum_{e_k^-\ge0} \phi(\epsilon_k^-)^{e_k^-}
\sum_{d_k^+\ge0} \phi(\delta_k^+)^{d_k^+} 
\sum_{d_k^-\ge0} \phi(\delta_k^-)^{d_k^-}
\\ &=
\prod_{k\ge1} 
\frac{1}{1-\phi(\epsilon_k^+)} \cdot
\frac{1}{1-\phi(\epsilon_k^-)} \cdot
\frac{1}{1-\phi(\delta_k^+)} \cdot
\frac{1}{1-\phi(\delta_k^-)}.
\end{align*}
This is equal to the right-hand side of (\ref{eq:AIII-generating-funct}),
since 
$\phi(\epsilon_k^\pm) = a^k b^k t_{2k}$,
$\phi(\delta_k^+) = a^k b^{k-1} t_{2k-1}$,
and $\phi(\delta_k^-) = a^{k-1} b^k t_{2k-1}$.
\end{proof}

\subsection{Types BDI, CI, CII, DIII}
\label{subsection:number.of.nilpotent.orbits.types.BDI-DIII}

We consider the symmetric pairs in Table~\ref{table:ssp.in.this.paper} in this paper.  
For other classical symmetric pairs, namely types AI and AII, 
the intersection $ \Gorbit_{\lambda} \cap \fraks $ is a single $ K $-orbit.  
So our problem becomes trivial.  
\begin{table}[htbp]
\caption{Table of symmetric pairs.}%
\label{table:ssp.in.this.paper}%
\hfil$
  \begin{array}{l@{\quad}l@{\;\;}l@{\;\;}l}
    \text{type} & (G,K) & n &
    \\ \hline
    \mathrm{AIII} & (GL_{p+q},GL_p \times GL_q) & p + q
    \\ 
    \mathrm{BDI} & (O_{p+q},O_p \times O_q) & p + q
    \\ 
    \mathrm{CI} & (Sp_{2p},GL_{p}) & 2 p
    \\ 
    \mathrm{CII} & (Sp_{p+q},Sp_{p}\times Sp_{q}) & p + q & (p, q : \text{even})
    \\ 
    \mathrm{DIII} & (O_{2p},GL_{p}) & 2 p
  \end{array}
$\hfil
\end{table}

In this table, 
for a symplectic group, we denote it by $ Sp_N $ in which $ N $ represents the dimension of the base symplectic space (or size of the matrices), 
hence $ N $ must be always even.  
Also in the case of type CI and DIII, we sometimes put $ q = p $ so that $ n = p + q $ holds.  
Thus, in the following, 
$ n $ always denotes the size of matrices in $ G $, 
and $ p $ or $ q $ denotes the size of the matrices of a simple factor of $ K $ (modulo its center).

Since the case of type AIII has been already treated, let us consider the other types, namely 
types BDI, CI, CII and DIII.  
For these symmetric pairs, 
nilpotent $G$-orbits on $\frakg$ and nilpotent $K$-orbits on $\fraks$ 
are parameterized by Young diagrams and signed Young diagrams
with suitable conditions, respectively.
In all these types, the conditions for signed Young diagrams
can be described by using primitives,
which consist rows of signed Young diagrams.
Primitives for these types are given in Table~\ref{table:primitives-of-syd} 
(\cite[Proposition~2]{Ohta.1991}; see also \cite[Proposition~2.2]{Trapa.2005}).
\begin{table}
  \caption{Primitives of signed Young diagrams.}
  \label{table:primitives-of-syd}
  \begin{tabular}{lllll}
    type & \multicolumn{2}{l}{primitives \; ($\{a,b\}=\{+,-\}$)}
    \\ \hline
    AIII 
    & $\begin{array}{c} ab \cdots ab \end{array}$(even),
    & $\begin{array}{c} ab \cdots ba \end{array}$(odd)
    \\ \hline
    BDI 
    & $\begin{array}{c} ab \cdots ba \end{array}$(odd),
    & $\begin{array}{c} ba \cdots ba \\ ab \cdots ab \end{array}$(even)
    \\ \hline
    CI
    & $\begin{array}{c} ab \cdots ab \end{array}$(even),
    & $\begin{array}{c} ab \cdots ba \\ ba \cdots ab \end{array}$(odd)
    \\ \hline
    CII 
    & $\begin{array}{l} ab \cdots ba \\ ab \cdots ba \end{array}$(odd),
    & $\begin{array}{l} ba \cdots ba \\ ab \cdots ab \end{array}$(even)
    \\ \hline
    DIII 
    & $\begin{array}{l} ba \cdots ba \\ ba \cdots ba \end{array}$(even),
    & $\begin{array}{l} ab \cdots ba \\ ba \cdots ab \end{array}$(odd)
  \end{tabular}
  \\
  ((even) or (odd) means the parity of the length.)
\end{table}

We denote by $\sydx(\lambda;p,q)$
the set of the signed Young diagrams for type X (X $=$ BDI, CI, CII, DIII) of shape $ \lambda $ 
with the convention that $ q = p $ in the case of type CI or DIII.  

Similarly we denote by $\yd_{\text{X}}(n)$
the set of the Young diagrams for type X.
Suppose we remove the signs in a signed Young diagram $ T $, and get a partition $ \lambda $, i.e., 
$ T \in \sydx(\lambda;p,q) $.  
Then a nilpotent $ K $-orbit $ \Korbit_T \subset \lie{s} $ 
corresponding to $ T $ generates a nilpotent $ G $-orbit 
$ \Gorbit_{\lambda} \subset \lie{g} $ 
corresponding to $ \lambda $.  
We get $ \yd_{\mathrm{X}}(n) $ in this way.

\begin{theorem}
\label{thm:BDI-CI-CII-DIII-generating-funct}
We have the generating functions of the numbers of 
the nilpotent $K$-orbits on $\fraks$ for the symmetric pairs of
types {\upshape{}BDI, CI, CII} and {\upshape{}DIII} as follows,
where the notation is the same as in 
Theorem~{\upshape\ref{thm:AIII-generating-funct}}.
\begin{rbenumerate}
\item
Write a partition $\lambda$ of $n=p+q$ of type {\upshape{}BDI}
as $\lambda = [1^{m_1} \cdot 2^{m_2} \cdots n^{m_n}] $
using the multiplicities $m_i$ of $i$.
Then we have
\begin{multline*}
  \sum_{\substack{p, \, q \ge 0,  \; \lambda \in \ydbdi(p+q)}}
  \# \sydbdi(\lambda;p,q) \, a^p \, b^q \, t_{\lambda}
  \\ 
  =
  \prod_{k=1}^\infty \dfrac{1}{1-a^{2k} b^{2k} t_{2k}^2} \cdot
  \dfrac{1}{1-a^{k-1} b^k t_{2k-1}} \cdot
  \dfrac{1}{1-a^k b^{k-1} t_{2k-1}}.
\end{multline*}

\item
Write a partition $\lambda$ of $n = 2 p$ of type {\upshape{}CI}
as $\lambda = [1^{m_1} \cdot 2^{m_2} \cdots n^{m_{n}}] $.
Then we have
\begin{multline*}
  \sum_{\substack{p \ge 0,  \; \lambda \in \ydci(2 p)}}
  \# \sydci(\lambda;p,p) \, a^p \, b^p \, t_{\lambda}
  \\ 
  =
  \prod_{k=1}^\infty 
  \dfrac{1}{(1-a^{k} b^{k} t_{2k})^2} \cdot
  \dfrac{1}{1-a^{2k-1} b^{2k-1} t_{2k-1}^2}.
\end{multline*}

\item
Write a partition $\lambda$ of $p+q$ of type {\upshape{}CII}
as $ \lambda = [1^{m_1} \cdot 2^{m_2} \cdots ] $.
Then the generating function of the number of nilpotent $K$-orbits
on $\fraks$ is given as follows.
\begin{multline*}
  \sum_{\substack{p, q \ge 0 \; (p, q : \text{even}),  \; \lambda \in \ydcii(p+q)}}
  \# \sydcii(\lambda;p,q) \, a^{p} \, b^{q} \, t_{\lambda}
  \\ 
  =
  \prod_{k=1}^\infty 
  \dfrac{1}{1-a^{2k-2} b^{2k} t_{2k-1}^2} \cdot
  \dfrac{1}{1-a^{2k} b^{2k-2} t_{2k-1}^2} \cdot
  \dfrac{1}{1-a^{2k} b^{2k} t_{2k}^2}. 
\end{multline*}

\item
Write a partition $\lambda$ of $n = 2 p$ of type {\upshape{}DIII}
as $\lambda = [1^{m_1} \cdot 2^{m_2} \cdots ] $.  
Then the generating function of the number of nilpotent $K$-orbits
on $\fraks$ is given as follows.
\begin{multline*}
  \sum_{\substack{p \ge 0,  \; \lambda \in \yddiii(2p)}}
  \# \syddiii(\lambda;p,p) \, a^{p} \, b^{p} \, t_{\lambda}
  \\ 
  =
  \prod_{k=1}^\infty 
  \dfrac{1}{(1-a^{2k} b^{2k} t_{2k}^2)^2} \cdot
  \dfrac{1}{1-a^{2k-1} b^{2k-1} t_{2k-1}^2}. 
\end{multline*}
\end{rbenumerate}
\end{theorem}
\begin{proof}
The proof is similar to that of Theorem~\ref{thm:AIII-generating-funct}.  
If a primitive contains $k_+$ $({+})$'s and $k_-$ $({-})$'s,
and consists of rows of lengths $l_1, l_2, \ldots, l_d$,
then the generating function has a factor
\begin{gather*}
\frac{1}{1 - a^{k_+} b^{k_-} t_{l_1} t_{l_2} \cdots t_{l_d}}.
\end{gather*}
Thus the formulas immediately follows from Table~\ref{table:primitives-of-syd}.
\end{proof}

\section{Combinatorial description of orbit graphs \\ for type AIII}
\label{section:structure.of.orbit.graph.AIII}

In this section, we consider a symmetric pair 
$ (G, K) = (GL_n, GL_p \times GL_q) $ 
of type AIII.

\subsection{Structure of orbit graph}
\label{subsec:AIII-orbit-graph}

To describe the whole structure of the orbit graph $ \Ograph_K(\Gorbit_{\lambda}) $, 
we prepare some notions.

The vertices of the graph $ \Ograph_K(\Gorbit_{\lambda}) $ is 
the set of nilpotent $ K $-orbits:
\begin{equation*}
V(\Ograph_K(\Gorbit_{\lambda})) = \{ \Korbit_T \mid T \in \syd(\lambda; p, q) \} .  
\end{equation*}
We realize these vertices as points in the Euclidean $ k $-space $ \RR^k $.  
To describe it, we denote $ \lambda $ in slightly different manner from the notation before, 
namely 
\begin{equation}
\label{eqn:lambda.expressed.with.multiplicity}
\begin{aligned}
\lambda 
&= ( i_1, \dots, i_1, i_2, \dots, i_2 , \dots, i_k, \dots, i_k) \\
&= (i_1^{m(i_1)}, i_2^{m(i_2)}, \dots , i_k^{m(i_k)} ) , \\
& \qquad\qquad
i_1 > i_2 > \cdots > i_k > 0 , \quad m(i_j) > 0 \;\; (1 \leq j \leq k), \quad
\end{aligned}
\end{equation}
where $ m(i) = m_{\lambda}(i) $ is the multiplicity of $ i $ among the parts of $ \lambda $,
which is a function in $ i $ and $\lambda$.  
If we pick $ T $ from $ \syd(\lambda; p, q) $, 
there are $ m(i) $ rows of length $ i $ in $ T $.  
Among those $ m(i) $ rows, 
some of them will begin with the box $\fboxplus$, and the others begin with the box $\fboxminus$.  
We denote the number of rows which begin with $\fboxplus$ by $ m^+(i) = m_T^+(i) $.  
We also write $ m^-(i) = m(i) - m^+(i) $, which is the number of rows of length $ i $ starting with box 
$\fboxminus$.  

Let us define a map
$\pi : V(\Ograph_K(\Gorbit_\lambda)) 
\simeq \syd(\lambda;p,q) \to \RR^k$ by
\begin{equation}
\label{eq:definition.of.pi}
\pi(T) = ( m^+(i_1) , m^+(i_2) , \dots , m^+(i_k) ) 
\in \ZZ_{\geq 0}^k \subset \RR^k.
\end{equation}
These $ m^+(i_r) $'s must satisfy obvious inequalities
\begin{equation*}
0 \leq m^+(i_r) \leq m(i_r) \qquad (1 \leq r \leq k) , 
\end{equation*}
and a \emph{parity condition} 
\begin{align}
p - q 
&= \sum_{i_r \text{ : odd}} ( m^+(i_r) - m^-(i_r) ) \notag \\
&= 2 \sum_{i_r \text{ : odd}} m^+(i_r) - \sum_{i_r \text{ : odd}} m(i_r) .
\label{eq:parity.condition.for.pi}
\end{align}
Note that the difference $ m^+(i_r) - m^-(i_r) $ only contributes to the difference $ p - q $ when 
the row length $ i_r $ is odd (if it is even, there are the same number of $+$'s and $-$'s in that row), 
hence the above parity condition.  

Conversely, if $ (a_1, \dots, a_k ) \in \ZZ_{\geq 0}^k $ satisfies 
\begin{equation*}
0 \leq a_r \leq m(i_r) \qquad (1 \leq r \leq k) , 
\end{equation*}
and the parity condition
\begin{equation*}
p - q = 2 \sum_{i_r \text{ : odd}} a_r - \sum_{i_r \text{ : odd}} m(i_r) ,
\end{equation*}
then $ (a_1, \dots, a_k) $ is in the image of the map $ \pi $, i.e., 
$ \pi(T) = (a_1, \dots, a_k) $ for some $ T \in \syd(\lambda; p,q) $.

Thus we are left to determine the edges of the orbit graph.  
We first recall Ohta's result on cover relations
(i.e., closure relation $\Korbit_S \subset \closure{\Korbit_T}$
with no orbits in-between)
of nilpotent $K$-orbits on $\fraks$ \cite[Lemma 5]{Ohta.1991}.

\begin{lemma}
\label{lem:AIII-cover-relation}
Let $\mu$ and $\lambda$ be partitions of $n=p+q$.
For signed Young diagrams 
$S \in \syd(\mu;p,q)$ and $T \in \syd(\lambda;p,q)$,
the corresponding nilpotent $K$-orbits $\Korbit_S$ and $\Korbit_T$ on $\fraks$
satisfy 
\begin{gather*}
  \Korbit_S \subset \closure{\Korbit_T},
  \quad
  \text{and there is no $K$-orbit in-between},
\end{gather*}
if and only if one of the following three conditions holds:
\begin{enumerate}
  \item [(i)]
  $ \overline{S} = \begin{array}{l}
    \overbrace{\cdots\cdots ab}^u \\
    \underbrace{\cdots ba}_v 
  \end{array} $,
  $ \overline{T} = \begin{array}{l}
    \overbrace{\cdots\cdots ba}^{u+1} \\
    \underbrace{\cdots ab}_{v-1} 
  \end{array} $
  \quad
  $(u \ge v \ge 1)$

  \item [(ii)]
  $ \overline{S} = \begin{array}{l}
    \overbrace{ba\cdots\cdots}^u \\
    \underbrace{ab\cdots}_v 
  \end{array} $,
  $ \overline{T} = \begin{array}{l}
    \overbrace{ab\cdots\cdots}^{u+1} \\
    \underbrace{ba\cdots}_{v-1} 
  \end{array} $
  \quad
  $(u \ge v \ge 1)$

  \item [(iii)]
  $ \overline{S} = \begin{array}{l}
    \overbrace{\cdots\cdots ba}^u \\
    \underbrace{\cdots ba}_v 
  \end{array} $,
  $ \overline{T} = \begin{array}{l}
    \overbrace{\cdots\cdots ba}^{u+2} \\
    \underbrace{\cdots ba}_{v-2} 
  \end{array} $
  \quad
  $(u \ge v \ge 2, \; \text{$u-v$: even})$,
\end{enumerate}
where $\{a,b\} = \{+,-\}$,
and $\overline{S}$ and $\overline{T}$ denote
the diagrams obtained by removing common rows from $S$ and $T$.
\end{lemma}

\begin{example}
The following is the graph of closure ordering of the nilpotent $K$-orbits
for the symmetric pair $(GL_6, GL_3 \times GL_3)$.
(See Figure~\ref{figure:closure.ordering.GL6,GL3XGL3}.)
%
\begin{figure}[htbp]
\includegraphics[scale=0.9]{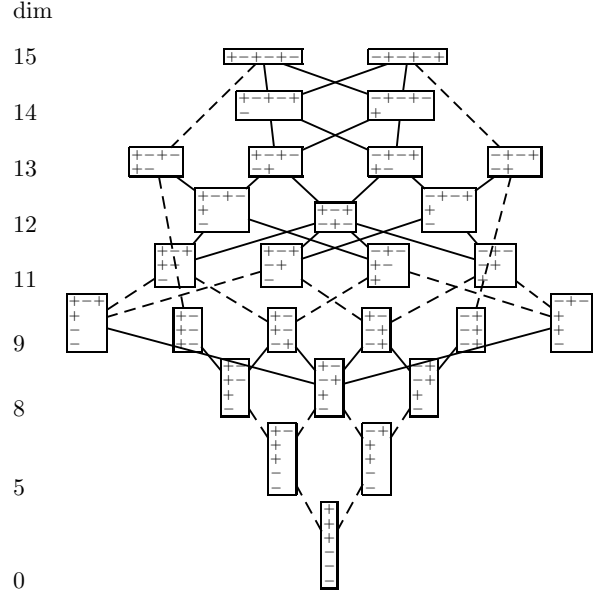}
\caption{Closure ordering: $ (GL_6, GL_3 \times GL_3) $.}
\label{figure:closure.ordering.GL6,GL3XGL3}
\end{figure}
\end{example}

\begin{example}
Figure~\ref{figure:closure.ordering.GL8,GL4XGL4} exhibits the graph of closure ordering of the nilpotent $K$-orbits
for the symmetric pair $(GL_8, GL_4 \times GL_4)$.

\begin{figure}[htbp]
\includegraphics[scale=0.8]{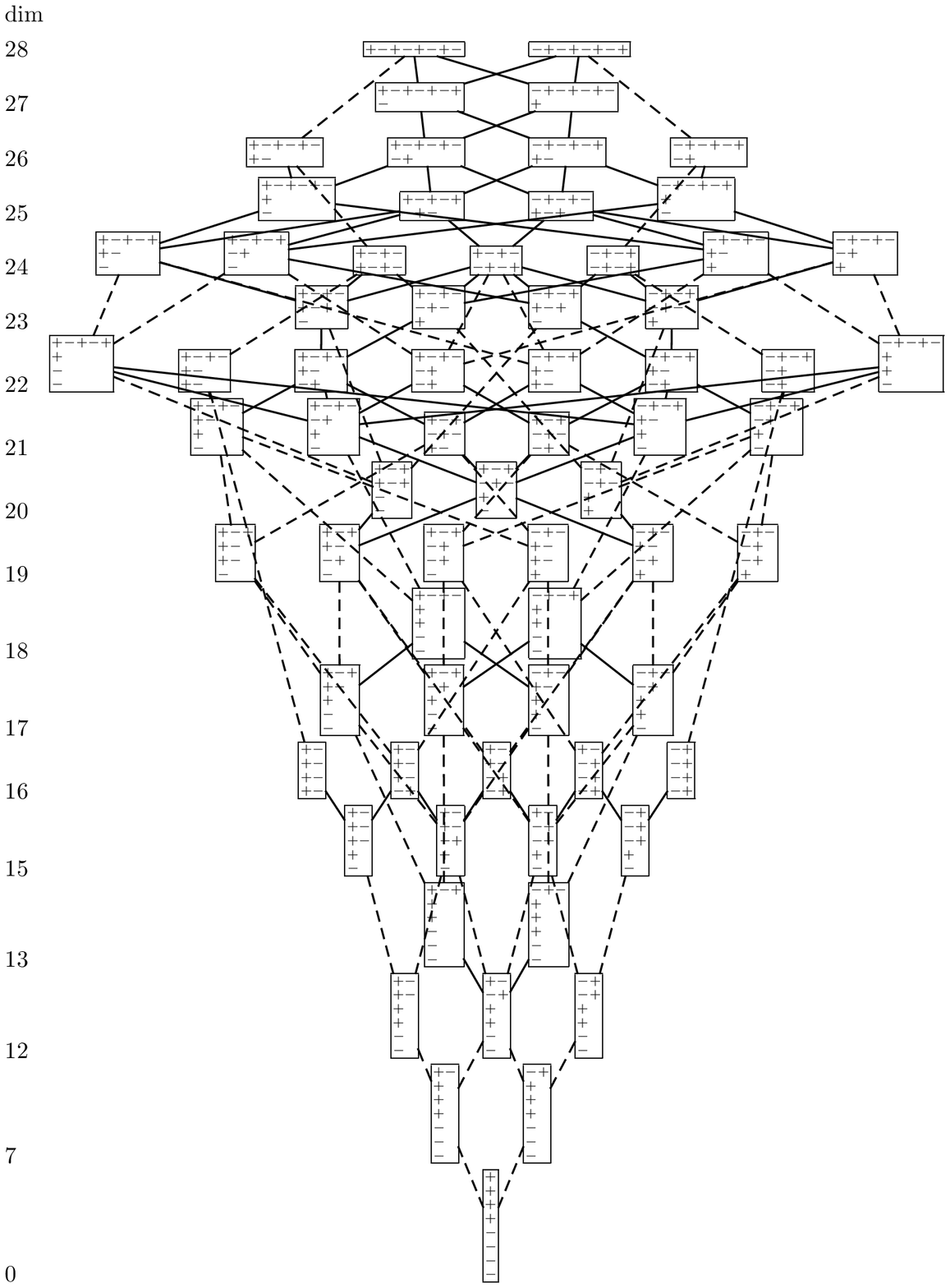}
\caption{Closure ordering: $ (GL_8, GL_4 \times GL_4) $.}
\label{figure:closure.ordering.GL8,GL4XGL4}
\end{figure}
\end{example}

In order to determine adjacency in codimension one,
we recall the dimension formula for $\Korbit_T$
(see \cite[Corollary~6.1.4]{Collingwood.McGovern.1993}, for example).

\begin{lemma}
\label{lem:AIII-dim-formula}
Let $\lambda$ be a partition of $n=p+q$, and $T \in \syd(\lambda;p,q)$.
The dimension of the nilpotent $K$-orbit $\Korbit_T$ is half of 
the dimension of the nilpotent $G$-orbit $\Gorbit_\lambda$,
and we have
\begin{gather*}
  \dim \Korbit_T = \frac{1}{2} \dim \Gorbit_\lambda =
  \frac{1}{2}\biggl( n^2 - \sum_{i=1}^r (\transpose{\lambda}_i)^2 \biggr),
\end{gather*}
where $\transpose{\lambda} = (\transpose{\lambda}_1, \transpose{\lambda}_2, \ldots, \transpose{\lambda}_r)$
denotes the transposed partition of $\lambda$.
\end{lemma}

Thus we obtain cover relations of nilpotent $K$-orbits on $\fraks$
of codimension one,
and hence the condition for two nilpotent $K$-orbits
$\Korbit_S$ and $\Korbit_T$ ($S,T \in \syd(\lambda;p,q)$)
to be adjacent in codimension one.

\begin{lemma}
\label{lem:AIII-closure-rel-codim=1}
Let $\mu$ and $\lambda$ be partitions of $n=p+q$, 
and take $S \in \syd(\mu;p,q)$ and $T \in \syd(\lambda;p,q)$ respectively.
Then $\Korbit_S \subset \closure{\Korbit_T}$
and $\dim \Korbit_S = \dim \Korbit_T - 1$
if and only if one of the following two conditions holds.  
\begin{itemize}
  \item [(i)]
  $ \overline{S} = \begin{array}{l}
    \overbrace{\cdots\cdots ab}^u \\
    \underbrace{\cdots ba}_v 
  \end{array} $, \quad
  $ \overline{T} = \begin{array}{l}
    \overbrace{\cdots\cdots ba}^{u+1} \\
    \underbrace{\cdots ab}_{v-1} 
  \end{array} $
  \quad
  $(u \ge v \ge 1)$,\\
  and $T$ has no rows of length $\ell = u,u-1,\ldots,v$.
  
  \item [(ii)]
  $ \overline{S} = \begin{array}{l}
    \overbrace{ba\cdots\cdots}^u \\
    \underbrace{ab\cdots}_v 
  \end{array} $,
    \quad
  $ \overline{T} = \begin{array}{l}
    \overbrace{ab\cdots\cdots}^{u+1} \\
    \underbrace{ba\cdots}_{v-1} 
  \end{array} $
  \quad
  $(u \ge v \ge 1)$,\\
  and $T$ has no rows of length $\ell = u,u-1,\ldots,v$.
\end{itemize}
\end{lemma}

\begin{proof}
Among three cases in Lemma~\ref{lem:AIII-cover-relation},
it turns out that in Case (iii) the codimension is always greater than one
by Lemma~\ref{lem:AIII-dim-formula}.
In Cases (i) and (ii) 
the codimensions are one
if and only if $T$ has no rows between two rows 
in $\overline{T}$.
\end{proof}

\begin{lemma}
\label{lem:AIII-adjacent-in-codim-1}
Let $\lambda$ be a partition of $n=p+q$, and $T, T' \in \syd(\lambda;p,q)$.
Then $\Korbit_T$ and $\Korbit_{T'}$ are adjacent in codimension one
if and only if one of the following two conditions holds.  
\begin{itemize}
  \item [(i)]
  $ \overline{T} = \begin{array}{l}
    \overbrace{ab\cdots\cdots ab}^{2u} \\
    \underbrace{ba\cdots ba}_{2v} 
  \end{array} $,\;\;
  $ \overline{T'} = \begin{array}{l}
    \overbrace{ba\cdots\cdots ba}^{2u} \\
    \underbrace{ab\cdots ab}_{2v} 
  \end{array} $
  \quad
  $(u > v \ge 0)$,\\
  and $\lambda$ has no rows of length $\ell = 2u-1,2u-2,\ldots,2v+1$.
  
  \item [(ii)]
  $ \overline{T} = \begin{array}{l}
    \overbrace{ab\cdots\cdots ba}^{2u+1} \\
    \underbrace{ba\cdots ab}_{2v+1} 
  \end{array} $,\;\;
  $ \overline{T'} = \begin{array}{l}
    \overbrace{ba\cdots\cdots ab}^{2u+1} \\
    \underbrace{ab\cdots ba}_{2v+1} 
  \end{array} $
  \quad
  $(u > v \ge 0)$,\\
  and $\lambda$ has no rows of length $\ell = 2u,2u-1,\ldots,2v+2$.
\end{itemize}
\end{lemma}

\begin{proof}
Suppose that there exists $S \in \syd(\mu;p,q)$ of shape $\mu \vdash n$
such that $\Korbit_S \subset \closure{\Korbit_{T}}$,
$\Korbit_S \subset \closure{\Korbit_{T'}}$,
and the codimension is equal to one.
Then the only possibility is that 
$\Korbit_S \subset \closure{\Korbit_{T}}$ satisfies (i) (resp.~(ii)),
and $\Korbit_S \subset \closure{\Korbit_{T'}}$ satisfies (ii) (resp.~(i))
in Lemma~\ref{lem:AIII-closure-rel-codim=1}.

Suppose the length of the first row of $ \overline{S} $ is odd.  
Since $ \overline{S} $ appears in (i) and (ii) in Lemma~\ref{lem:AIII-closure-rel-codim=1} 
at the same time, the signatures $ a, b $ in (i) and those in (ii) must coincide.  
Thus the length of the second row is also odd, which leads us to the case (i) 
in the present lemma.  
Similarly, if the length of the first row of $ \overline{S} $ is even, 
in Lemma~\ref{lem:AIII-closure-rel-codim=1}, the signatures $ a, b $ in (i) and those in (ii) must be interchanged.  
So the length of the second row is also even, which leads us to the case (ii) 
in the present lemma.  
\end{proof}

\begin{theorem}[Description of orbit graph]
\label{thm:AIII-description-of-graph}
Let $ \lambda $ be a partition of $ n $, 
and $ \syd(\lambda; p, q) $ the set of signed Young diagrams 
with signature $ (p, q) $.  
Recall the map $ \pi: \syd(\lambda;p,q) \to \RR^k$ from Equation \eqref{eq:definition.of.pi}, 
where $ k $ is the number of parts of $ \lambda $ of different length $($see Equation \eqref{eqn:lambda.expressed.with.multiplicity}$)$.  

The structure of the orbit graph 
$ \Ograph_K(\Gorbit_{\lambda}) $ is described as follows.  
The vertices are $ \{ \Korbit_T \mid T \in \syd(\lambda; p, q) \} $ and, 
for two vertices $ \Korbit_{T} $ and $ \Korbit_{T'} $, 
there is an edge if and only if 
$ \pi(T) - \pi(T') $ belongs to 
\begin{equation*}
 \{ \pm (e_r - e_{r + 1} ) \mid 1 \leq r \leq k - 1 \} \cup \{ \pm e_k \} .  
\end{equation*}
Here $ e_r $ denotes a fundamental unit vector which has $ 1 $ in the $ r $-th coordinate 
and $ 0 $ elsewhere.
\end{theorem}

\begin{proof}
By the definition of 
$\pi: \syd(\lambda;p,q) \to \RR^k$
and Lemma~\ref{lem:AIII-adjacent-in-codim-1},
we immediately have the description of the edges.
Note that the case where $v=0$ 
in Case (i) of Lemma~\ref{lem:AIII-adjacent-in-codim-1}
corresponds to the edges $\pm e_k$.
\end{proof}

\begin{example}
\label{ex:orbit-graphs}
(1)
Consider the shape $\lambda = (6,4,4,2,2)$ 
and signature $(p,q) = (9,9)$.
The following is (the image under $\pi$ of)
the graph of $\syd(\lambda;p,q)$,
where dotted lines are just for help to see the structure.

\begin{figure}[htbp]
\includegraphics[scale=0.9]{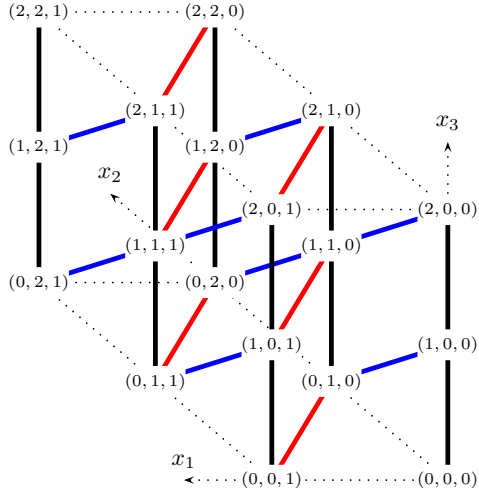}
\caption{Orbit graph for $ \lambda = (6,4,4,2,2)$.}
\label{figure:orbit.graph.64422.AIII}
\end{figure}

(2)
Consider the shape $\lambda = (4,3,3,1,1)$ 
and signature $(p,q) = (6,6)$.
Figure~\ref{figure:orbit.graph.43311.AIII} is (the image under $\pi$ of)
the graph of $\syd(\lambda;p,q)$.  
Again dotted lines are just for help to see the structure.

\noindent\hfil
\begin{figure}[htbp]
\includegraphics[scale=0.7]{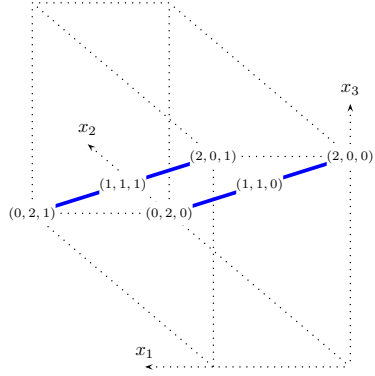}
\caption{Orbit graph for $ \lambda = (4, 3, 3, 1, 1)$.}
\label{figure:orbit.graph.43311.AIII}
\end{figure}
\end{example}

From this theorem, we can give a complete system of representatives 
of the connected components of $\Ograph_K(\Gorbit_{\lambda})$ 
in algorithmic way.  
The idea of getting such a representative is to start 
from an orbit $ \Korbit_T $ from a connected component, then 
to move rows in $ T \in \syd(\lambda; p,q) $ 
beginning with $\fboxplus$ as upper as possible within the connected component containing 
$ \Korbit_T $.  

To describe these representatives explicitly, let us introduce some notation.  

Let $ \lambda $ be a partition of $ n $ with length 
$ \ell = \ell(\lambda) $ and put $ \lambda_{\ell+1} = 0 $.  
Define $ k_0 = 0 < k_1 < k_2 < \cdots < k_m  $ by 
\begin{equation}
\label{eq:k1k2...km}
  \{ k_1, k_2, \ldots, k_m \}
  = \{ j \mid 
    1 \le j \le \ell, \; 
    \text{$\lambda_j - \lambda_{j+1}$ is odd}
    \} ,
\end{equation}
and put 
\begin{equation}
\label{eq:parity.condition}
\begin{aligned}
&
P(\lambda; p,q) =
\{ \mbfp = (p_1,\ldots,p_m) \in \ZZ_{\geq 0}^m \mid 
\text{$ \mbfp $ satisfies 
$ (\ast) $} \}
\\
&
\qquad
(\ast)
\left\{ 
\text{\parbox{0.5\textwidth}{
      $0 \le p_s \le k_s - k_{s-1} $ $(1\le s \le m)$,\\[1.5ex]
      $2 \!\! \textstyle\sum\limits_{\text{$\lambda_{k_s} : $ odd}} p_s - \#(\text{odd parts}) = p-q $.
  }}
\right.
\end{aligned}
\end{equation}
If there is no odd part in $ \lambda $, then we formally put $ m = 1 , k_1 = 0 $ and $ P(\lambda; p, q) = \{ (0) \} $, otherwise we get $ k_1 > 0 $.  
For $ \mbfp = (p_1,\ldots,p_m) \in P(\lambda; p, q) $, we construct a signed Young diagram 
$T \in \syd(\lambda; p,q)$ in such a way that 
$j$-th row begins with $\fboxplus$ if and only if 
$ k_{s -1} < j  \le k_{s - 1} + p_s $ for some $1 \le s \le m$.
Then the parity condition in $ (\ast) $ for $\sum_{\text{$\lambda_{k_s}$: odd}} p_s$ 
assures that $ T $ has indeed the desired signature $ (p, q) $ 
(see Equation \eqref{eq:parity.condition.for.pi}).  
Again, if there is no odd part in $ \lambda $, we associate $ (0) \in P(\lambda; p, q) $ 
with a signed Young diagram $ T $ in which every row starts with $\fboxminus$.  
In this case it is necessary that $ p = q = n/2 $ holds (thus $ n $ must be even in this case).

\begin{lemma}
\label{lem:complete-representatives}
With the above notation, the set 
\begin{align*}
  \left\{
    T \in \syd(\lambda; p,q) \mid 
    \text{$ T $ constructed from 
      $\mbfp \in P(\lambda; p,q)$}
  \right\},
\end{align*}
gives a complete system of representatives of connected components of the 
graph $ \Ograph_K(\Gorbit_{\lambda}) $.  
\end{lemma}

\begin{proof}
This lemma follows easily from Theorem~\ref{thm:AIII-description-of-graph}.
More precisely,
this complete system corresponds to the greatest signed Young diagrams 
with respect to the total order defined by
\begin{align*}
T_1 \ge T_2 \Leftrightarrow
\begin{cases}
(1) & \text{the number of $\fboxplus$ in $r_1(T_1)$ $<$ that in $r_1(T_2)$},
\\
(2) & \text{or the number of $\fboxplus$ in $r_1(T_1)$ $=$ that in $r_1(T_2)$},
\\
&\text{and $r_1(T_1) \ge_{\text{lex}} r_1(T_2)$},
\end{cases}
\end{align*}
where $r_1(T)$ denotes the first column of $T$,
and $\ge_{\text{lex}}$ denotes the lexicographic order 
with $\fboxplus$ $>$ $\fboxminus$.
\end{proof}

\subsection{Product of graph}
\label{subsec:AIII-product-of-graph}
The orbit graph $\Ograph_K(\Gorbit_\lambda)$  
associated to the set of signed Young diagrams $\syd(\lambda;p,q)$ is presented as 
a disjoint union of products of 
basic building blocks.
There are two kinds of the basic building blocks
$A(\mbfm; \rho)$ and $C(\mbfm)$ defined below.
Take a partition $\lambda$ of $n=p+q$,
and write
$\lambda = (i_1^{m(i_1)}, i_2^{m(i_2)}, \ldots, i_k^{m(i_k)})$
using multiplicities 
(see Equation \eqref{eqn:lambda.expressed.with.multiplicity}).

Let us use the notation in~(\ref{eq:k1k2...km}) 
and Lemma~\ref{lem:complete-representatives}.  
For $ 1 \leq s \leq m $, we put 
$ r_s $ to be the number of different parts of $ \lambda $ between 
the first row and the $ k_s $-th row (we count the $ k_s $-th row also).  
Then we have an increasing sequence $r_1 < r_2 < \cdots < r_m \leq k $.  
Recall that $ k $ is the number of different parts of $ \lambda $.  
Here $ r_m = k $ holds if the last part of $ \lambda $ is odd.  
If the last part of $ \lambda $ is even, 
$i_{r_{m}+1}, i_{r_{m}+2}, \ldots, i_k $
are different even row lengths at the tail of $\lambda$.
See Example~\ref{example:product.of.graph}, where these numbers $ r_s $'s as well as $ k_s $'s are given for several $ \lambda $'s.

For a collection of non-negative integers $ \mbfm = (m_1, m_2, \dots, m_{\ell}) $ and $\rho$, 
we define connected graphs
$A(\mbfm; \rho)$ and $C(\mbfm)$ 
as follows.
The vertices of $A(\mbfm; \rho)$  are given by
\begin{gather*}
  \left\{
    (a_1, a_2, \ldots, a_{\ell}) \in \ZZ^{\ell}
    \Bigm|
    \begin{array}{l}
      0 \le a_s \le m_s \; (1 \le s \le {\ell}), \\
      a_1 + a_2 + \cdots + a_{\ell} = \rho
    \end{array}
  \right\},
\end{gather*}
and the edge between
$(a_1, a_2, \ldots, a_{\ell})$ and $(b_1, b_2, \ldots, b_{\ell})$
exists if and only if
\begin{gather*}
(a_1, a_2, \ldots, a_{\ell}) - (b_1, b_2, \ldots, b_{\ell}) = \pm (e_s - e_{s+1})
\end{gather*}
for some $s = 1,2,\ldots,{\ell}-1$.
The vertices of $C(\mbfm)$ is 
\begin{gather*}
\{ (a_1, a_2, \ldots, a_{\ell}) \in \ZZ^{\ell} \mid
  0 \le a_s \le m_s \; (1 \le s \le {\ell}) \},
\end{gather*}
and the edge between
$(a_1, a_2, \ldots, a_{\ell})$ and $(b_1, b_2, \ldots, b_{\ell})$
exists if and only if
\begin{align*}
(a_1, a_2, \ldots, a_{\ell}) &- (b_1, b_2, \ldots, b_{\ell}) \\
&= 
\begin{cases}
  \pm (e_s - e_{s+1}) \; (s = 1,2,\ldots,{\ell}-1), \; \text{or} \\
  \pm e_{\ell}.
\end{cases}
\end{align*}
If the parameter is empty, we set $C(\emptyset)$ 
to be the graph of a single point with no edge.
For example, $A(2,1;1)$ and $C(1,2)$ are as follows:

\hfil
\parbox[b]{.3\textwidth}{\includegraphics{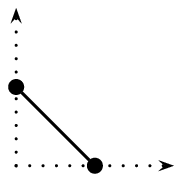}\\[1ex] $ A(2,1;1) $}
\hfil
\parbox[b]{.3\textwidth}{\includegraphics{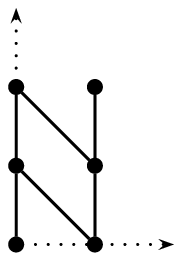}\\[1ex] $ C(1,2) $}



\begin{theorem}
\label{thm:product-of-graph}
Under the above notation, the orbit graph $\Ograph_K(\Gorbit_\lambda)$
for a partition $\lambda$ of $n=p+q$ can be presented 
as a disjoint union of direct products of simple connected graphs as
\begin{gather*}
\Ograph_K(\Gorbit_\lambda) \simeq
\coprod\nolimits_{\mbfp \in P(\lambda;p,q)} Z_{\mbfp}, \\
\intertext{where, if $ r_m < k $, the product $ Z_{\mbfp} $ is defined by}
\begin{aligned}
Z_{\mbfp} = 
A(m(i_1), & m(i_2),\ldots,m(i_{r_1}); p_1) \\
&\times A(m(i_{r_1+1}),m(i_{r_1+2}),\ldots,m(i_{r_2}); p_2) \times \cdots \\
&\times A(m(i_{r_{m-1}+1}),m(i_{r_{m-1}+2}),\ldots,m(i_{r_m}); p_m) \\
&\times C(m(i_{r_{m}+1}),m(i_{r_{m}+2}),\ldots,m(i_{k})), 
\end{aligned}
\intertext{and, if $ r_m = k $, }
\begin{aligned}
Z_{\mbfp} = 
A(m(i_1), & m(i_2),\ldots,m(i_{r_1}); p_1) \\
&\times A(m(i_{r_1+1}),m(i_{r_1+2}),\ldots,m(i_{r_2}); p_2) \times \cdots \\
&\times A(m(i_{r_{m-1}+1}),m(i_{r_{m-1}+2}),\ldots,m(i_{r_m}); p_m) .
\end{aligned}
\end{gather*}
\end{theorem}

\begin{proof}
The set of vertices of the orbit graph $\Ograph_K(\Gorbit_\lambda)$ is 
in one-to-one correspondence with the set of signed Young diagrams $\syd(\lambda;p,q)$,
and, if $ i_m $ is strictly smaller than $ k $, its image under the map $\pi:\syd(\lambda;p,q) \to \RR^k$ is 
{\small
\begin{align}
\label{eq:decomposition-of-vertices-of-graph}
\begin{split}
  & \left\{
    (a_1, a_2, \ldots, a_k) \in \ZZ^k
    \Bigm|
    \begin{array}{l}
      0 \le a_s \le m(i_s) \; (1 \le s \le k), \\
      2\textstyle\sum\limits_{\text{$i_s$: odd}} a_s - \#(\text{odd parts}) = p-q
    \end{array}
  \right\}
  \\ &=
  \coprod_{\mbfp \in P(\lambda;p,q)}
  \left\{
    (a_1, a_2, \ldots, a_k) \in \ZZ^k
    \Big|
    \begin{array}{l}
      0 \le a_s \le m(i_s) \; (1 \le s \le k), \\
      a_{r_{t-1}+1} + \cdots + a_{r_t} = p_t \\
      \hfill (1 \le t \le m)
    \end{array}
  \right\}
  \\ &\simeq
  \coprod_{\mbfp \in P(\lambda;p,q)}
  \prod_{t=1}^m
  V\bigl(A(m(i_{r_{t-1}+1}), \ldots, m(i_{r_t}); p_t)\bigr) 
  \\ & \hspace*{.4\textwidth}
\times
  V\bigl(C(m(i_{r_{m}+1}), \ldots, m(i_{k}))\bigr), 
\end{split}
\end{align}
}where we put $r_0 = 0$, and $ V(\Gamma) $ denotes the set of vertices of a graph $ \Gamma $.
If $ r_m = k $, then the last term in the last equality will not appear.

Since the edges of $\Ograph_K(\Gorbit_\lambda)$ are of the form
$\pm(e_s - e_{s+1})$ or $\pm e_k$ ($s=1,2,\ldots,k-1$ and $i_s-i_{s+1}$ is even),
every edge sits inside some factor of the right-hand side of
(\ref{eq:decomposition-of-vertices-of-graph}).
Therefore (\ref{eq:decomposition-of-vertices-of-graph})
turns out to be a disjoint union of direct products
not only as sets but also as graphs.
\end{proof}

\begin{example}
\label{example:product.of.graph}
(1)
Let $\lambda=(6,4,4,2,2) = (6, 4^2, 2^2)$ be a partition of $ 18 $ and $(p,q)=(9,9)$.
\begin{align*}
&
(i_1, i_2, i_3) = (6, 4, 2), \quad (m(i_1),m(i_2),m(i_3))=(1,2,2), \quad k = 3, \\
&
(k_0, k_1) = (0, 0), \quad m = 1, \quad r_1 = 0, \quad\\
&
P(\lambda; p, q) = \{ (0) \} .
\end{align*}
Thus $\Ograph_K(\Gorbit_\lambda) \simeq C(1,2,2)$
as given in Example~\ref{ex:orbit-graphs}~(1).

(2)
Let $\lambda=(4,3,3,1,1)=(4, 3^2, 1^2) $ 
and $(p,q)=(6,6)$.
\begin{align*}
&
(i_1, i_2, i_3) = (4, 3, 1), \quad (m(i_1),m(i_2),m(i_3))=(1,2,2), \quad k = 3, \\
&
(k_0, k_1, k_2) = (0, 1, 5), \quad m = 2, \quad (r_1, r_2) = (1, 3), \\
&
P(\lambda; p, q) = \{ (0, 2), (1, 2) \} .
\end{align*}
So we have
\begin{align*}
\Ograph_K(\Gorbit_\lambda) &\simeq 
\coprod_{(p_1, p_2)}
A(1;p_1) \times A(2,2;p_2)
\\ & =
A(1;0) \times A(2,2;2) \amalg A(1;1) \times A(2,2;2)
\\
& \simeq \raisebox{-2.5ex}{\includegraphics{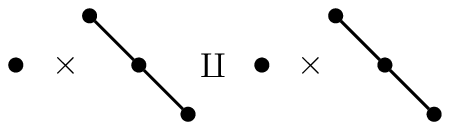}}
\\[1ex] 
& \simeq \raisebox{-2.5ex}{\includegraphics{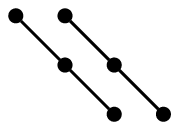}}
\end{align*}
as given in Example~\ref{ex:orbit-graphs}~(2).

(3)
Let $\lambda=(9,9,8,8,6,5,4,2,2)=(9^2, 8^2, 6, 5, 4, 2^2) $ and $(p,q)=(27,26)$.
\begin{align*}
&
(i_1, i_2, \dots, i_6) = (9, 8, 6, 5, 4, 2), \quad \\
&
(m(i_1),m(i_2),\dots,m(i_6))=(2,2,1,1,1,2), \quad k = 6, \\
&
(k_0, k_1, k_2, k_3) = (0, 2, 5, 6), \quad m = 3, \quad (r_1, r_2, r_3) = (1, 3, 4), \\
&
P(\lambda; p, q) = \{ (p_1, p_2, p_3) \mid p_1 \in [0, 2], p_2 \in [0, 3], p_3 \in [0, 1], p_1 + p_3 = 2 \} .
\end{align*}
Notice that the parity condition for $ \mbfp \in P(\lambda; p, q) $ reads as 
$ 2 (p_1 + p_3) - 3 = 27 - 26 $, so we get $ p_1 + p_3 = 2 $.  
Thus we have
\begin{align*}
&
\Ograph_K(\Gorbit_\lambda) \\
&= 
\coprod_{(p_1, p_2, p_3)}
A(2;p_1) \times A(2,1;p_2) \times A(1;p_3) \times C(1,2)
\\ & \simeq
\Bigl(A(2;1) \times A(1;1) \amalg A(2;2) \times A(1;0)\Bigr) \! \times \!\!
\textstyle\coprod\limits_{p_2=0}^3 A(2,1;p_2) \times C(1,2)
\\ 
& \simeq \raisebox{-2.5ex}{\includegraphics{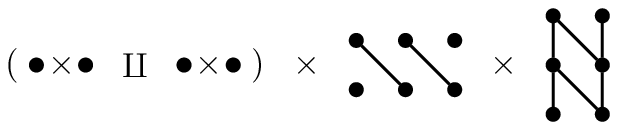}}
\\[1ex] 
& \simeq \raisebox{-2.5ex}{\includegraphics{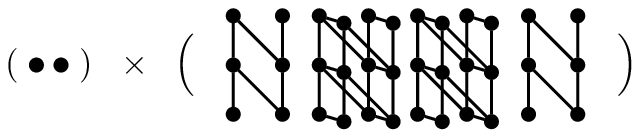}}
\\[1ex] 
& \simeq \raisebox{-2.5ex}{\includegraphics{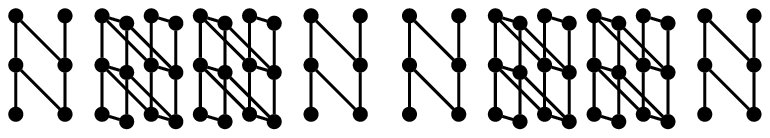}}
\\
\end{align*}
\end{example}

\subsection{Induction of subgraphs}
\label{subsec:AIII-induction.of.subgraph}

Let us consider the following operation on the partitions.  
We identify the partitions with Young diagrams in standard way.  
Given a Young diagram (or a partition) $ \lambda $, 
we remove two successive \emph{columns} of the same length from $ \lambda $ (if they exist), 
and we get $ \lambda' $.  
To explain this operation in another way, let us consider the transposed partition $ \mu = \transpose{\lambda} $.  
If $ \mu = ( \mu_1, \dots , \mu_{\ell'} ) $ has a pair of repeated parts, i.e., 
if $ \mu = ( \mu_1, \dots , \mu_i, \mu_{i + 1} , \dots , \mu_{\ell'} ) $ with $ \mu_i = \mu_{i + 1} $, 
we remove that pair, and then take the transpose again.  So we get 
\begin{equation}
\label{eq:elimination.of.columns}
\lambda' = \transpose{
  ( \mu_1, \dots , \hat{\mu_i}, \hat{\mu_{i + 1}} , \dots , \mu_{\ell'} ) }, 
\end{equation}
where $ \hat{\cdot} $ means elimination.  

\begin{lemma}
\label{lem:number-of-connected-components-coincide}
Let $\lambda$ and $\lambda'$ be as above,
and $h$ the height of the columns removed from $\lambda$.
Then the number of connected components of $\Ograph_K(\Gorbit_{\lambda})$
coincides with that of $\Ograph_{K'}(\Gporbit_{\lambda'})$,
where $(G,K) = (GL_n, GL_p \times GL_q)$, 
and $(G',K') = (GL_{n-2h}, GL_{p'} \times GL_{q'}) $ 
with 
$ p' = p - h $ and $ q' = q - h $.

Note that if $n = 2h$, then $\lambda'$ is the empty Young diagram,
and  $\Ograph_{K'}(\Gporbit_{\lambda'})$ should be considered as the one-point graph 
(with no edges) 
whose vertex is parameterized by the empty signed Young diagram.
\end{lemma}

\begin{proof}
The number $ 0 < k_1 < k_2 < \cdots < k_m$ for $\lambda$ given in Equation \eqref{eq:k1k2...km} 
are the same as those for $\lambda'$,
since the parities of the row lengths are the same
for $\lambda$ and $\lambda'$.
By the same reason the number of the odd parts is the same 
for $\lambda$ and $\lambda'$.
Together with $p-q = (p-h)-(q-h) = p' - q'$,
it turns out that the set $P(\lambda'; p', q')$,
which parameterizes the connected components of $\Ograph_{K'}(\Gporbit_{\lambda'})$ 
is equal to $P(\lambda; p,q)$.
Hence the number of the connected components of $\Ograph_K(\Gorbit_{\lambda})$ 
coincides with that of $\Ograph_{K'}(\Gporbit_{\lambda'})$. 
\end{proof}

Let us refine the lemma above, which helps us to understand the connected components more concretely.  
Actually, we describe the connected components of $ \Ograph_K(\Gorbit_{\lambda}) $ 
in terms of those of $ \Ograph_{K'}(\Gporbit_{\lambda'}) $.  
To do so, we need some notation.

Let $ Z' $ be a full subgraph of $ \Ograph_{K'}(\Gporbit_{\lambda'}) $.  
For each vertex $ \Kporbit_{T'} $ in $ Z' $, 
we construct several nilpotent $ K $-orbits $ \{ \Korbit_T \}_T $ as follows.  
Since $ \lambda' $ is contained in $ \lambda $ (as a Young diagram, in the left and upper justified manner), 
we can put the signed Young diagram $ T' $ inside the shape $ \lambda $.  
In other words, we fill $ \pm $'s in $ \lambda' \subset \lambda $ in such a way that it recovers $ T' $.  
If $ T' $ has several rows of the same length, we allow every possible permutations of such rows.  
After that, we fill $ \pm $'s in $ \lambda/\lambda' $ in every possible way, 
which is compatible with $ T' $.  

\begin{example}
(1)\ 
Let us consider the case where $ (p,q) = (8,7) $, $ (p', q') = (3,2) $, and 
$$ \lambda' = (2^2, 1) \subset \lambda = (4^2, 3, 2^2) . $$
Pick 
$ T' \in \syd(\lambda'; 3,2) $ below, and we get a set of signed Young diagrams in $ \syd(\lambda; 8,7) $ 
as follows.
\begin{gather*}
T' = \text{\tiny$\young(+-,-+,+)$} \in \syd(\lambda'; 3,2) \rightsquigarrow \{ T \in \syd(\lambda; 8,7) \} 
\\[2ex]
T = 
\makebox[0pt][l]{\raisebox{-.5ex}{\text{\Yboxdim10pt\Ylinethick1pt\yng(2,2,1)}}}
\text{\tiny$\young(+-+-,-+-+,+-+,+-,+-)$} , \quad
\makebox[0pt][l]{\raisebox{-.5ex}{\text{\Yboxdim10pt\Ylinethick1pt\yng(2,2,1)}}}
\text{\tiny$\young(+-+-,-+-+,+-+,+-,-+)$} , \quad
\makebox[0pt][l]{\raisebox{-.5ex}{\text{\Yboxdim10pt\Ylinethick1pt\yng(2,2,1)}}}
\text{\tiny$\young(+-+-,-+-+,+-+,-+,-+)$} 
\end{gather*}
(2)\ 
Similarly we give an example where $ (p,q) = (7,5) $, $ (p', q') = (4,2) $, and 
$ \lambda' = (2, 1^4) \subset \lambda = (4, 3^2, 1^2) $.  
Let us consider 
$ T' \in \syd(\lambda'; 4,2) $ below.  
Note that we can reorder the tail of $ T' $ as we like.  
\begin{equation*}
T' = \text{\tiny$\young(+-,+,+,+,-)$} = \text{\tiny$\young(+-,+,-,+,+)$} 
\end{equation*}
Then we obtain 
$ \{ T \in \syd(\lambda; 7,5) \} $ from $ T' $ as follows.
\begin{equation*}
T = 
\makebox[0pt][l]{\raisebox{-3.85ex}{\text{\Yboxdim10pt\Ylinethick1pt\yng(2,1,1,1,1)}}}
\text{\tiny$\young(+-+-,+-+,+-+,+,-)$} , \quad
\makebox[0pt][l]{\raisebox{-3.85ex}{\text{\Yboxdim10pt\Ylinethick1pt\yng(2,1,1,1,1)}}}
\text{\tiny$\young(+-+-,+-+,-+-,+,+)$} 
\end{equation*}
\end{example}

We get several signed Young diagrams of the shape $ \lambda $ in this way.  
We repeat this procedure for each vertex $ \Kporbit_{T'} $ of $ Z' $.  
Collecting all the signed Young diagrams thus obtained from $ Z' $, 
we finally get a subset 
$ \ind(Z') \subset \syd(\lambda; p, q) $ or a subset of nilpotent $ K $-orbits contained in $ \Gorbit_{\lambda} \cap \lie{s} $.  
(Since $ Z' $ is a graph, we should write $ \ind(V(Z')) $ instead of $ \ind(Z') $, but we prefer this simpler notation.)  
We denote a full subgraph of $ \Ograph_K(\Gorbit_{\lambda}) $ with the vertices in 
$ \ind(Z') $ by $ \gind_{(G',K')}^{(G, K)}(Z') $ or simply by 
$ \gind(Z') $.

\begin{lemma}
\label{lemma:gind.of.connected.components}
Let $ \lambda $ and $ \lambda' $ be as above and we use the notation in Lemma 
\ref{lem:number-of-connected-components-coincide}.  
If $ Z' $ is a connected component of $ \Ograph_{K'}(\Gporbit_{\lambda'}) $, 
then $ \gind(Z') $ is a connected component of $ \Ograph_K(\Gorbit_{\lambda}) $.  
This correspondence establishes a bijection between the connected components of 
$ \Ograph_{K'}(\Gporbit_{\lambda'}) $ and those of $ \Ograph_K(\Gorbit_{\lambda}) $.
\end{lemma}

\begin{proof} 
Note that any $ T \in \syd(\lambda; p, q) $ is contained in $ \ind(\{T'\}) $ for 
some $ T' \in \syd(\lambda'; p', q') $.  
Also, for two signed Young diagrams $ T' \neq T'' \in \syd(\lambda'; p', q') $, 
it is immediate to see that $ \ind(\{T'\}) \cap \ind(\{T''\}) = \emptyset $.  
Thus it is sufficient to prove that 
$ \gind(Z') \subset \Ograph_K(\Gorbit_{\lambda}) $ is connected.  
In fact, if we can prove that $ \gind(Z') $ is connected, 
we have a well-defined surjective map from the connected components of $ \Ograph_{K'}(\Gporbit_{\lambda}) $ 
to those of $ \Ograph_K(\Gorbit_{\lambda}) $.  
Since the number of connected components are equal by Lemma \ref{lem:number-of-connected-components-coincide}, 
this map must be bijective.  
By the arguments above, $ \ind(Z') $ covers all the vertices of $ \Ograph_K(\Gorbit_{\lambda}) $ when 
$ Z' $ moves connected components of $ \Ograph_{K'}(\Gporbit_{\lambda}) $.  
This means that $ \gind(Z') $ must be a connected component.

So let us prove that $ \gind(Z') $ is connected.

Take $ T' \in \syd(\lambda'; p', q') $, where $ p' = p - h $ and $ q' = q - h $.  
First, we will prove that $ \gind(\{T'\}) $ is connected.  

We write $ \mu = \transpose{\lambda} = (\mu_1, \mu_2, \dots, \mu_{\ell_1}) $, and 
\begin{equation*}
\lambda' = \transpose{
  ( \mu_1, \dots , \hat{\mu_i}, \hat{\mu_{i + 1}} , \dots , \mu_{\ell_1} ) }, 
  \quad
  h = \mu_i = \mu_{i + 1}
\end{equation*}
as in Equation \eqref{eq:elimination.of.columns}.  
Here, without loss of generality, we can assume that the removed columns are at the rightmost position 
among the columns of the same length $ h $, i.e., 
$ \mu_{i+1} > \mu_{i + 2} $ with the convention $ \mu_{\ell_1} > \mu_{\ell_1 + 1} = 0 $.  
Then there are three possibilities: 
(i)\ 
$ i > 1 $ and 
$ \mu_{i - 1} = \mu_i $; 
(ii)\ 
$ i > 1 $ and 
$ \mu_{i - 1} > \mu_i $; 
(iii)\ 
$ i = 1 $, i.e., we remove first two columns.
Let us recall the map $ \pi $ in Equation \eqref{eq:definition.of.pi}, and 
choose an arbitrary $ T \in \ind(\{T'\}) $.

Case (i).  
In this case, it is easy to see that there is a unique choice for $ T $, and 
$ \ind(\{T'\}) $ is one point.  So it is connected.

Case (ii).  
In this case, we have $ \mu_{i - 1} > \mu_i = \mu_{i + 1} > \mu_{i + 2} $.  
As in Equation \eqref{eqn:lambda.expressed.with.multiplicity},
we write 
\begin{equation*}
\begin{aligned}
\lambda 
&= ( i_1, \dots, i_1, i_2, \dots, i_2 , \dots, i_k, \dots, i_k) \\
&= (i_1^{\nu_1}, i_2^{\nu_2}, \dots , i_k^{\nu_k} ) , \\
& \qquad\qquad
i_1 > i_2 > \cdots > i_k > 0 , \quad \nu_r > 0 \;\; (1 \leq r \leq k) .
\end{aligned}
\end{equation*}
If we remove two columns of the same length $ \mu_i = \mu_{i + 1} $ from $ \lambda $, we get
\begin{equation*}
\begin{aligned}
\lambda' 
&= ({i_1'}^{\nu_1'}, {i_2'}^{\nu_2'}, \dots , {i_{k - 1}'}^{\nu_{k- 1}'} ) , \\
& \qquad\qquad
i_1' > i_2' > \cdots > i_{k - 1}' > 0 , \quad \nu_r' > 0 \;\; (1 \leq r \leq k - 1) .
\end{aligned}
\end{equation*}
Since $ \lambda = \lambda' + (2^h) \; (h = \mu_i = \mu_{i + 1}) $, there exists $ 1 \leq j < k $ such that 
$ i_j = i_{j + 1} + 2 $ and 
\begin{align*}
& 
\begin{cases}
i_r = i_r' + 2 \\
\nu_r = \nu_r' 
\end{cases}
\qquad
(1 \leq r \leq j - 1) , 
\\ & 
\begin{cases}
i_{j + 1} = i_j' \\
\nu_j + \nu_{j + 1} = \nu_j' 
\end{cases}
,
\\ & 
\begin{cases}
i_r = i_{r - 1}' \\
\nu_r = \nu_{r - 1}' 
\end{cases}
\qquad
(j + 2 \leq r \leq k) .
\end{align*}
Fix $ T \in \ind(\{T'\}) $ and we write 
\begin{equation*}
\pi(T) = (m_T^+(i_1), m_T^+(i_2), \dots, m_T^+(i_k)) 
=: (a_1, a_2, \dots, a_k) \in \ZZ_{\geq 0}^k 
\end{equation*}
and
\begin{equation*}
\pi(T') = (m_{T'}^+(i_1'), \dots, m_{T'}^+(i_{k - 1}')) 
=: (b_1, \dots, b_{k - 1}) \in \ZZ_{\geq 0}^{k - 1} .
\end{equation*}
Then by the definition of the map $ \pi $ and the construction of the signed Young diagram $ T $, 
we get 
\begin{equation*}
(b_1, \dots, b_{k - 1}) 
= (a_1, \dots, a_{j - 1}, a_j + a_{j + 1}, a_{j + 2}, \dots, a_k) .
\end{equation*}
Thus we conclude that 
\begin{align*}
\{ \pi(T) \mid 
&
T \in \ind(\{T'\}) \} \\
&= \Biggl\{ (b_1, \dots, b_{j - 1}, a_j, a_{j + 1}, b_{j + 1}, \dots, b_{k - 1}) \Biggm| 
\begin{aligned}
&
a_j + a_{j + 1} = b_j \\ 
&
0 \leq a_j \leq \nu_j \\
&
0 \leq a_{j + 1} \leq \nu_{j + 1} 
\end{aligned}
\Biggr\} .
\end{align*}
Note that the parity condition \eqref{eq:parity.condition.for.pi} is automatically satisfied since 
$ \pi(T') = (b_1, \dots, b_{k - 1}) $ satisfies it, and $ i_j $ and $ i_{j + 1} $ have the same parity.  
Now it is clear that 
$ \{ \pi(T) \mid T \in \ind(\{T'\}) \} $ constitutes a segment in the direction of $ \pm (e_j - e_{j+1}) $, 
hence $ \gind(\{T'\}) $ is connected.

Case (iii).  
In this case, we must have 
\begin{equation*}
\begin{aligned}
\lambda 
&= (i_1^{\nu_1}, i_2^{\nu_2}, \dots , i_{k- 1}^{\nu_{k - 1}}, 2^{\nu_k} ) , \\
& \qquad\qquad
i_1 > i_2 > \cdots > i_{k - 1} > 2 , \quad \nu_r > 0 \;\; (1 \leq r \leq k) .
\end{aligned}
\end{equation*}
We remove first two columns from $ \lambda $ and get
\begin{equation*}
\begin{aligned}
\lambda' 
&= ((i_1 - 2)^{\nu_1}, (i_2 - 2)^{\nu_2}, \dots , (i_{k- 1} - 2)^{\nu_{k - 1}}) .
\end{aligned}
\end{equation*}
If we denote 
$ \pi(T') = (b_1, \dots, b_{k - 1}) \in \ZZ_{\geq 0}^{k - 1} $ as above, 
we conclude that 
\begin{align*}
\{ \pi(T) \mid T \in \ind(\{T'\}) \} 
= \{ (b_1, \dots, b_{k - 1}, a_k) \mid 
0 \leq a_k \leq \nu_k \} .
\end{align*}
This set also constitutes a segment in the direction of $ \pm e_k $, 
hence $ \gind(\{T'\}) $ is connected.

Next, we prove that 
if $ T' $ and $ T'' $ in $ \syd(\lambda'; p', q') $ are adjacent in codimension one, 
then there are $ T_1 \in \ind(\{T'\}) $ and $ T_2 \in \ind(\{T''\}) $ which are adjacent in 
$ \syd(\lambda; p, q) $.  
We also prove this by case-analysis, so we divide the proof into three cases (i)--(iii) introduced above.  
These cases depend only on $ \lambda $ and $ \lambda' $, not depending on individual $ T' \in \syd(\lambda'; p', q') $.  

Case (i).  
In this case, there is only one signed Young diagram $ T_1 $ belonging to $ \ind(\{T'\}) $ for any $ T' $.  
It is easy to check that $ \pi(T_1) = \pi(T') $.  
The same is true for $ \{ T_2 \} = \ind(T'') $.  
Thus we know $ \pi(T_1) - \pi(T_2) = \pi(T') - \pi(T'') $, and 
this gives the edge in the orbit graph realized in $ \RR^k $.
So the claim obviously holds.

Case (ii).  
Let 
$ \pi(T') = (b_1, \dots, b_{k - 1}) $ and 
$ \pi(T'') = (d_1, \dots, d_{k - 1}) $ as above.  
Then 
$ \pi(T_1) = (b_1, \dots, b_{j - 1}, a_j, a_{j + 1}, b_{j + 1}, \dots, b_{k - 1}) $ for certain integers $ a_j, a_{j + 1} $ 
with the property 
$ a_j + a_{j + 1} = b_j $ and 
$ 
0 \leq a_j \leq \nu_j ,
0 \leq a_{j + 1} \leq \nu_{j + 1} 
$.  
Similarly  
$ \pi(T_2) = (d_1, \dots, d_{j - 1}, c_j, c_{j + 1}, d_{j + 1}, \dots, d_{k - 1}) $ 
with 
$ c_j + c_{j + 1} = d_j $ and 
$ 
0 \leq c_j \leq \nu_j ,
0 \leq c_{j + 1} \leq \nu_{j + 1} 
$.
Let us assume that $ \pi(T') = \pi(T'') + (e_r - e_{r + 1}) $ for certain $ r $, i.e., 
assume that $ T' $ and $ T'' $ are connected by the edge corresponding to $ e_r - e_{r + 1} $.

If $ r \neq j - 1, j $, then $ b_j = d_j $ holds, and we can take $ (a_j, a_{j + 1}) = (c_j, c_{j + 1}) $.  
Thus $ T_1 $ and $ T_2 $ are connected by the edge corresponding to $ e_r - e_{r + 1} $.  

If $ r = j - 1 $, then $ (b_{j-1} , b_j) = (d_{j - 1} + 1, d_j - 1) $ and all the other $ b $'s and $ d $'s coincide with each other.  
Since $ c_j + c_{j + 1} = d_j \geq 1 $, we can assume that $ c_j \geq 1 $.  
If we put 
$ (a_j, a_{j + 1}) = (c_j - 1, c_{j + 1}) $, 
clearly $ T_1 $ and $ T_2 $ are connected by the edge corresponding to $ e_{j - 1} - e_j $.  
The case of $ r = j $ can be treated similarly.

Next, we assume that $ T' $ and $ T'' $ are connected by the edge corresponding to $ e_{k - 1} $.  
If $ j \neq k - 1 $, then we can take $ (a_j, a_{j + 1}) = (c_j, c_{j + 1}) $ and conclude that 
$ T_1 $ and $ T_2 $ are connected by the edge $ e_{k} $.  
If $ j = k - 1 $, we have 
$ b_j = b_{k -1} = d_{k - 1} + 1 \geq 1 $.  
Since $ a_{k - 1} + a_k = b_{k - 1} $, we can choose $ a_k > 0 $ and put 
$ (d_{k - 1}, d_k) = (a_{k - 1}, a_k - 1) $.  
Then, clearly $ T_1 $ and $ T_2 $ are connected by the edge $ e_k $.

Case (iii).  
Assume that 
\begin{align*}
\pi(T') &= (b_1, \dots, b_{k - 1}) , &
\pi(T_1) &= (b_1, \dots, b_{k - 1}, a_k) \quad (0 \leq a_k \leq \nu_k) , \\
\pi(T'') &= (d_1, \dots, d_{k - 1}) , &
\pi(T_2) &= (d_1, \dots, d_{k - 1}, c_k) \quad (0 \leq c_k \leq \nu_k) .  
\end{align*}
If $ T' $ and $ T'' $ are connected by the edge $ e_r - e_{r + 1} \; (r < k - 1) $, we can take $ a_k = c_k $ above, and 
conclude that $ T_1 $ and $ T_2 $ are also connected by the edge $ e_r - e_{r + 1} $.  

If $ T' $ and $ T'' $ are connected by the edge $ e_{k - 1} $, 
we can take $ a_k =0, c_k = 1 $ above, and 
conclude that $ T_1 $ and $ T_2 $ are also connected by the edge $ e_{k -1} - e_{k} $.  
\end{proof}

In Lemma~\ref{lemma:gind.of.connected.components}
we have proved the correspondence between
the connected components of $\Ograph_{K'}(\Gporbit_{\lambda'})$
and that of $\Ograph_{K}(\Gporbit_{\lambda})$,
where $\lambda'$ is a Young diagram obtained from $\lambda$
by removing two successive columns of the same length.
Repeating this operation we get the correspondence between
$\Ograph_{K''}(\Gpporbit_{\lambda''})$ and $\Ograph_{K}(\Gporbit_{\lambda})$,
where $\lambda''$ is obtained from $\lambda$
by removing two successive columns of the same length for finitely many times.
We denote this correspondence by the same notation
such as $\gind(Z'') = \gind_{(G'',K'')}^{(G, K)}(Z'')$.
It follows from the definition of $\gind$
that $\gind_{(G'',K'')}^{(G, K)}(Z'')$ is independent
of the order of removing the columns.

\subsection{Number of connected components}
\label{subsec:AIII-number.of.components}

If we remove pairs of the columns with the same length from $ \lambda $ repeatedly, 
then we will finally reach a Young diagram $ \rho $ with columns of different lengths.  
Lemma \ref{lem:number-of-connected-components-coincide} tells that 
the orbit graph $ \Ograph_{K'}(\Gporbit_{\rho}) $ has 
the same number of connected components as that of $ \Ograph_K(\Gorbit_{\lambda}) $.  
Therefore, to answer Problem~\ref{prob:general-case}~\eqref{prob:item:number.of.connected.components.of.graph},
it suffices to consider the Young diagrams with columns of different lengths.

\begin{theorem}
\label{thm:number.of.connected.components.of.orbit.graph}
\begin{thmenumerate}
\item
\label{thm:number.of.connected.components.of.orbit.graph:item:1}
The orbit graph $ \Ograph_K(\Gorbit_{\lambda}) $ consists of a single vertex if and only if 
{\upshape{(a)}} the parts in $ \lambda $ are all odd; and 
{\upshape{(b)}} $ \ell(\lambda) = |p - q| $ or $ \lambda = (r^{\ell}) $ for some odd $ r $.
\item
\label{thm:number.of.connected.components.of.orbit.graph:item:2}
The orbit graph $ \Ograph_K(\Gorbit_{\lambda}) $ has no edges 
if and only if each column length of $\lambda$ occurs odd times or it consists of a single vertex.
In particular, if $ \lambda $ has distinct column lengths, i.e., 
if the transposed partition $ \transpose{\lambda} $
has distinct parts,
then $ \Ograph_K(\Gorbit_{\lambda}) $ has no edge.

\item
\label{thm:number.of.connected.components.of.orbit.graph:item:3}
Assume that $ \lambda $ has distinct column lengths.  
In this case the number of the connected components 
{\upshape(}i.e., the number of the vertices{\upshape)}
of the orbit graph $\Ograph_K(\Gorbit_\lambda)$ is given by
\begin{align}
\label{eq:the-number-of-connected-components}
\prod_{\substack{ 1\le s\le m\\ \text{$\lambda_{k_s} : $ odd}}} (1+t+\cdots+t^{k_s-k_{s-1}}) 
\biggr|_{t^d}
\times
\prod_{\substack{ 1\le s\le m\\ \text{$\lambda_{k_s} : $ even}}} (1 + k_s - k_{s-1}),
\end{align}
where $ k_0 = 0 < k_1 < k_2 < \cdots < k_m$ are the {\upshape(}distinct{\upshape)} column lengths
of $\lambda$,
$f(t) \big|_{t^d}$ denotes the coefficient of $t^d$,
and $d$ is the number given by 
\begin{equation}
d := \frac{p-q+\#(\text{odd parts of $ \lambda $})}{2} .
\end{equation}
%
If $ d $ is not an integer, then there is no signed Young diagram of shape $ \lambda $ with signature $ (p, q) $.
\end{thmenumerate}
\end{theorem}

\begin{proof}
\eqref{thm:number.of.connected.components.of.orbit.graph:item:1}
If there is an even part in $ \lambda $, then clearly we have more than two signed Young diagrams of the same shape $ \lambda $ 
(the even part can start with the both $ {+}/{-} $ signs).  
So the parts in $ \lambda $ should be odd.  
Now we assume $ p \geq q $.  The case where $ p < q $ can be treated similarly.  
Since all the parts in $ \lambda $ are odd, the parity condition 
\eqref{eq:parity.condition} becomes
\begin{equation}
\label{eq:parity.condition.for.single.vertex}
\textstyle\sum\limits_{\text{$\lambda_{k_s} : $ odd}} p_s = \dfrac{1}{2} \bigl( \ell(\lambda) + p - q  \bigr) .
\end{equation}
Since there should be a unique choice of $ (p_s)_{s = 1}^{m} $, if $ m \neq 1 $, 
all $ p_s $'s must attain the largest possible value, namely $ p_s = k_s - k_{s - 1} $.  
Then the left hand side of 
\eqref{eq:parity.condition.for.single.vertex} 
is equal to $ k_m = \ell(\lambda) $, and we get $ \ell(\lambda) = p - q $.  
On the other hand, $ m = 1 $ forces a unique column length so that we have $ \lambda = (r^{\ell}) $ for some $ r $.

\eqref{thm:number.of.connected.components.of.orbit.graph:item:2}
Let us assume the orbit graph has more than two vertices.  
The partition $\lambda$ has a column length that occurs even times 
if and only if 
\begin{itemize}
\item [(i)]
there are two successive row lengths $i_s$ and $i_{s+1}$ of the same parity, or 
\item [(ii)]
the smallest part of $\lambda$ is even. 
\end{itemize}
By 
Lemma~\ref{lem:AIII-adjacent-in-codim-1}, 
this condition is equivalent to the condition 
that the orbit graph $\Ograph_K(\Gorbit_\lambda)$ has an edge provided that there are at least two vertices.  
Hence $\Ograph_K(\Gorbit_\lambda)$ has no edges  
if and only if each column length of $\lambda$ occurs odd times.

\eqref{thm:number.of.connected.components.of.orbit.graph:item:3}
Note that the numbers $k_s$'s are the same as $k_s$'s defined in Equation \eqref{eq:k1k2...km}.  
Thus it suffices to count the elements in $P(\lambda; p,q)$ defined just after Equation \eqref{eq:k1k2...km}.  

If $ \lambda_{k_s} $ is even, 
$p_s$ can be any integer contained in the interval $[0, k_s-k_{s-1}]$.
Therefore the number of choices is 
equal to the second product of \eqref{eq:the-number-of-connected-components}.
If $ \lambda_{k_s} $ is odd, 
we can choose integers $p_s$ in $[0, k_s-k_{s-1}]$ subject to the relation
\begin{gather*}
  \sum_{\text{$s$:odd}} p_s = \frac{p-q+\#(\text{odd parts})}{2} = d.
\end{gather*}
Note that the integer $d$ coincides with the number of the rows of odd length 
beginning with $\fboxplus$.
Therefore the number of choices for $p_s$ ($\lambda_{k_s}$: odd) is 
the coefficient of $t^d$ in
\begin{gather*}
  \prod_{\substack{ 1\le s\le m\\ \text{$\lambda_{k_s}$: odd}}} (1+t+\cdots+t^{k_s-k_{s-1}}) .
\end{gather*}
Thus we have the desired formula.
\end{proof}

From this theorem,
the condition for an orbit graph to be connected is immediate.

\begin{corollary}
\label{cor:AIII-condition-of-connectedness}
For a nilpotent $G$-orbit $\Gorbit_\lambda$ in $\frakg$,
the graph $\Ograph_K(\Gorbit_\lambda)$ is connected
if and only if there exists $ 0 \leq r \leq \ell = \ell(\lambda) $ such that 
\begin{align} 
  \label{eq:condition-for-connectedness}
  \text{
    $\lambda_1, \ldots, \lambda_r$ are odd,
    and $\lambda_{r+1}, \ldots, \lambda_{\ell}$ are even. }
\end{align}
Since $ r $ can be $ 0 $ or $ \ell $, 
we allow the cases where all the $ \lambda_i $'s are even, or where 
they are all odd.
\end{corollary}

\begin{proof}
The second product of (\ref{eq:the-number-of-connected-components})
is equal to one if and only if the product is empty,
namely, there is no even parts in $\lambda$
except for successive even parts at the tail of 
$\lambda = (\lambda_1, \lambda_2, \ldots, \lambda_{\ell})$.
Thus (\ref{eq:condition-for-connectedness}) is the necessary condition.

If (\ref{eq:condition-for-connectedness}) is satisfied,
then there is at most one factor in the first product of 
(\ref{eq:the-number-of-connected-components}),
and the second product is equal to one.
Therefore (\ref{eq:condition-for-connectedness}) is 
also the sufficient condition.
\end{proof}

For a nilpotent $G$-orbit $\Gorbit_{\lambda}$
there is a corresponding weighted Dynkin diagram
(\cite[Corollary~3.2.15]{Collingwood.McGovern.1993}),
which is a Dynkin diagram with vertices labeled by 0, 1 or 2.
A nilpotent $G$-orbit which corresponds to a weighted Dynkin diagram
with even labels (0 or 2) only
is called an \emph{even} nilpotent orbit.

It is known 
that a nilpotent orbit $ \Gorbit_{\lambda} $ is even if and only if 
all the parts of $ \lambda $ have the same parity, 
and this evenness condition is the same in the other classical cases 
(see \cite[\S~5.3]{Collingwood.McGovern.1993}).  
So we have

\begin{corollary}
\label{cor:even.nilpotent.AIII}
Let us consider the symmetric pair of type {\upshape{AIII}}.  
If a nilpotent orbit $ \Gorbit $ is even, 
the orbit graph $ \Ograph_K(\Gorbit) $ is connected.
\end{corollary}

\section{Orbit graphs for classical symmetric pairs}
\label{section:other.classical.symmetric.pairs}

For symmetric pairs of types other than AIII,
we have similar results on 
the structure of orbit graphs,
induction of subgraphs
and the number of connected components of orbit graphs.

\subsection{Structure of orbit graphs}

As to the structure of orbit graphs,
we need information on 
\begin{itemize}
\item 
vertices of the graph, i.e., the classification of nilpotent $ K $-orbits by the set of signed Young diagrams.
This can be deduced from Table~\ref{table:primitives-of-syd}.
\item 
 closure relations of nilpotent $K$-orbits on $\fraks$ of codimension one 
(Lemma~\ref{lem:BDI-CI-CII-DIII-codim-1-inclusion}).
This can be deduced from the following.
\begin{itemize}
\item 
cover relations of nilpotent $K$-orbits,
which are given by Ohta \cite[Table V]{Ohta.1991} 
quoted in Tables~\ref{table:Ohta-Table-V-1} and 
\ref{table:Ohta-Table-V-2}.
\item 
dimension formulas for nilpotent $K$-orbits
(Lemma~\ref{lem:dimension-formula}).
\end{itemize}
\item 
edges of the graph, i.e., 
the condition when two nilpotent $K$-orbits are adjacent in codimension one
(Lemma~\ref{lem:BDI-CI-CII-DIII-adjacent-relation}).
\end{itemize}

\bgroup\tiny
\begin{table}
\caption{Cover relations of nilpotent $K$-orbits on $\fraks$.}
\label{table:Ohta-Table-V-1}
$\begin{array}{l|lllllll}
\hline
\multicolumn{2}{l}{\hfil\text{DIII}}
\\ \hline
& \hspace*{6ex} \overline{S} \hspace*{12ex} \overline{T} 
\\ \hline
\makebox[0pt][r]{(1)} & 
\begin{array}{l} \abDDba[2u-1]\\\baDDab\\\abDba\\\baDab[2v-1] \end{array}
\begin{array}{l} \abDDab[2u  ]\\\abDDab\\\baDba\\\baDba[2v-2] \end{array}
(u\ge v \ge 1)
\\
\makebox[0pt][r]{(2)} &
\begin{array}{l} \baDDba[2u] \\\baDDba \\\abDba \\\baDab[2v-1] \end{array}
\begin{array}{l} \abDDba[2u+1]\\\baDDab\\\baDba\\\baDba[2v-2] \end{array}
(u\ge v \ge 1)
\\
\makebox[0pt][r]{(3)} &
\begin{array}{l} \abDDba[2u-1]\\\baDDab\\\baDba\\\baDba[2v] \end{array}
\begin{array}{l} \baDDba[2u]\\\baDDba\\\abDba\\\baDab[2v-1] \end{array}
(u\ge v \ge 1)
\\
\makebox[0pt][r]{(4)} &
\begin{array}{l} \abDDab[2u]\\\abDDab\\\baDba\\\baDba[2v] \end{array}
\begin{array}{l} \abDDba[2u+1]\\\baDDab\\\abDba\\\baDab[2v-1] \end{array}
(u\ge v \ge 1)
\\
\makebox[0pt][r]{(5)} &
\begin{array}{l} \baDDba[2u]\\\baDDba\\\baDba\\\baDba[2v] \end{array}
\begin{array}{l} \baDDba[2u+2]\\\baDDba\\\baDba\\\baDba[2v-2] \end{array}
(u\ge v \ge 1)
\end{array}$
\hfill
$\begin{array}{l|lllllll}
\hline
\multicolumn{2}{l}{\hfil\text{CII}}
\\ \hline
& \hspace*{6ex} \overline{S} \hspace*{12ex} \overline{T} 
\\ \hline
\makebox[0pt][r]{(1)} & 
\begin{array}{l} \baDDba[2u]\\\abDDab\\\baDba\\\abDab[2v] \end{array}
\begin{array}{l} \baDDab[2u+1]\\\baDDab\\\abDba\\\abDba[2v-1] \end{array}
(u\ge v \ge 1)
\\
\makebox[0pt][r]{(2)} & 
\begin{array}{l} \abDDba[2u+1]\\\abDDba\\\baDba\\\abDab[2v] \end{array}
\begin{array}{l} \baDDba[2u+2]\\\abDDab\\\abDba\\\abDba[2v-1] \end{array}
(u\ge v \ge 1)
\\
\makebox[0pt][r]{(3)} & 
\begin{array}{l} \baDDba[2u]\\\abDDab\\\abDba\\\abDba[2v+1] \end{array}
\begin{array}{l} \abDDba[2u+1]\\\abDDba\\\baDba\\\abDab[2v] \end{array}
(u\ge v \ge 0)
\\
\makebox[0pt][r]{(4)} & 
\begin{array}{l} \baDDab[2u+1]\\\baDDab\\\abDba\\\abDba[2v+1] \end{array}
\begin{array}{l} \baDDba[2u+2]\\\abDDab\\\baDba\\\abDab[2v] \end{array}
(u\ge v \ge 0)
\\
\makebox[0pt][r]{(5)} & 
\begin{array}{l} \abDDba[2u+1]\\\abDDba\\\abDba\\\abDba[2v+1] \end{array}
\begin{array}{l} \abDDba[2u+3]\\\abDDba\\\abDba\\\abDba[2v-1] \end{array}
(u\ge v \ge 1)
\end{array}$
\end{table}

\begin{table}
\caption{Cover relations of nilpotent $K$-orbits on $\fraks$ (continued).}
\label{table:Ohta-Table-V-2}
$\begin{array}{l|lllllll}
\hline
\multicolumn{2}{l}{\hfil\text{CI}}
\\ \hline
& \hspace*{6ex} \overline{S} \hspace*{12ex} \overline{T} 
\\ \hline
\makebox[0pt][r]{(1)} & 
\begin{array}{l} \abDDba[2u-1]\\\underbrace{\baDDab}_{2u-1} \end{array}
\begin{array}{l} \abDDab[2u]\\\baDba[2u-2] \end{array}
(u=v \ge 1)
\\
\makebox[0pt][r]{(2)} & 
\begin{array}{l} \baDDba[2u]\\\baDba[2v] \end{array}
\begin{array}{l} \baDDba[2u+2]\\\baDba[2v-2] \end{array}
(u \ge v \ge 1)
\\
\makebox[0pt][r]{(3)} & 
\begin{array}{l} \baDDba[2u]\\\abDab[2v] \end{array}
\begin{array}{l} \baDDba[2u+2]\\\abDab[2v-2] \end{array}
(u \ge v \ge 1)
\\
\makebox[0pt][r]{(4)} & 
\begin{array}{l} \baDDba[2u]\\\abDDab\\\baDba[2v] \end{array}
\begin{array}{l} \abDDba[2u+1]\\\baDDab\\\baDba[2v-2] \end{array}
(u \ge v \ge 1)
\\
\makebox[0pt][r]{(5)} & 
\begin{array}{l} \baDDba[2u]\\\baDba\\\abDab[2v] \end{array}
\begin{array}{l} \baDDba[2u+2]\\\baDab\\\abDba[2v-1] \end{array}
(u \ge v \ge 1)
\\
\makebox[0pt][r]{(6)} & 
\begin{array}{l} \baDDba[2u]\\\abDDab\\\baDba\\\abDab[2v] \end{array}
\begin{array}{l} \baDDab[2u+1]\\\abDDba\\\baDab\\\abDba[2v-1] \end{array}
(u \ge v \ge 1)
\\
\makebox[0pt][r]{(7)} & 
\begin{array}{l} \abDDab[2u]\\\abDDab\\\baDba\\\baDba[2v] \end{array}
\begin{array}{l} \abDDba[2u+1]\\\baDDab\\\abDba\\\baDab[2v-1] \end{array}
(u \ge v \ge 1)
\\
\makebox[0pt][r]{(8)} & 
\begin{array}{l} \abDDab[2u]\\\baDba[2v] \end{array}
\begin{array}{l} \baDDba[2u+2]\\\abDab[2v-2] \end{array}
(u \ge v \ge 1)
\\
\makebox[0pt][r]{(9)} & 
\begin{array}{l} \abDDab[2u]\\\baDba\\\baDba[2v] \end{array}
\begin{array}{l} \baDDba[2u+2]\\\abDba\\\baDab[2v-1] \end{array}
(u \ge v \ge 1)
\\
\makebox[0pt][r]{(10)} & 
\begin{array}{l} \abDDab[2u]\\\abDDab\\\baDba[2v] \end{array}
\begin{array}{l} \abDDba[2u+1]\\\baDDab\\\abDab[2v-2] \end{array}
(u \ge v \ge 1)
\end{array}$
\hfill
$\begin{array}{l|lllllll}
\hline
\multicolumn{2}{l}{\hfil\text{BDI}}
\\ \hline
& \hspace*{6ex} \overline{S} \hspace*{12ex} \overline{T} 
\\ \hline
\makebox[0pt][r]{(1)} & 
\begin{array}{l} \baDDba[2u]\\\underbrace{\abDDab}_{2u} \end{array}
\begin{array}{l} \baDDab[2u+1]\\\abDba[2u-1] \end{array}
(u = v \ge 1)
\\
\makebox[0pt][r]{(2)} & 
\begin{array}{l} \abDDba[2u+1]\\\abDba[2v+1] \end{array}
\begin{array}{l} \abDDba[2u+3]\\\abDba[2v-1] \end{array}
(u \ge v \ge 1)
\\
\makebox[0pt][r]{(3)} & 
\begin{array}{l} \abDDba[2u+1]\\\baDab[2v+1] \end{array}
\begin{array}{l} \abDDba[2u+3]\\\baDab[2v-1] \end{array}
(u \ge v \ge 1)
\\
\makebox[0pt][r]{(4)} & 
\begin{array}{l} \abDDba[2u+1]\\\baDDab\\\abDba[2v+1] \end{array}
\begin{array}{l} \baDDba[2u+2]\\\abDDab\\\abDba[2v-1] \end{array}
(u \ge v \ge 1)
\\
\makebox[0pt][r]{(5)} & 
\begin{array}{l} \abDDba[2u+1]\\\abDba\\\baDab[2v+1] \end{array}
\begin{array}{l} \abDDba[2u+3]\\\abDab\\\baDba[2v] \end{array}
(u \ge v \ge 0)
\\
\makebox[0pt][r]{(6)} & 
\begin{array}{l} \abDDba[2u+1]\\\baDDab\\\abDba\\\baDab[2v+1] \end{array}
\begin{array}{l} \abDDab[2u+2]\\\baDDba\\\abDab\\\baDba[2v] \end{array}
(u \ge v \ge 0)
\\
\makebox[0pt][r]{(7)} & 
\begin{array}{l} \baDDab[2u+1]\\\baDDab\\\abDba\\\abDba[2v+1] \end{array}
\begin{array}{l} \baDDba[2u+2]\\\abDDab\\\baDba\\\abDab[2v] \end{array}
(u \ge v \ge 0)
\\
\makebox[0pt][r]{(8)} & 
\begin{array}{l} \baDDab[2u+1]\\\abDba[2v+1] \end{array}
\begin{array}{l} \abDDba[2u+3]\\\baDab[2v-1] \end{array}
(u \ge v \ge 1)
\\
\makebox[0pt][r]{(9)} & 
\begin{array}{l} \baDDab[2u+1]\\\abDba\\\abDba[2v+1] \end{array}
\begin{array}{l} \abDDba[2u+3]\\\baDba\\\abDab[2v] \end{array}
(u \ge v \ge 0)
\\
\makebox[0pt][r]{(10)} & 
\begin{array}{l} \baDDab[2u+1]\\\baDDab\\\abDba[2v+1] \end{array}
\begin{array}{l} \baDDba[2u+2]\\\abDDab\\\baDab[2v-1] \end{array}
(u \ge v \ge 1)
\end{array}$
\end{table}
\egroup

\begin{lemma}
\label{lem:dimension-formula}
For a symmetric pair of type 
${\rm X} = {\rm BDI}, {\rm CI}, {\rm CII}$ or ${\rm DIII}$,
let $\lambda \in \yd_{\rm X}(n)$ be a partition of $n = p+q$,
and $T \in \sydx(\lambda;p,q)$ a signed Young diagram of type {\upshape{}X}.
Recall that we put $ q = p $ in the case of type {\upshape CI} or {\upshape DIII}.  
Then we have
\begin{align*}
  & \dim \Korbit_T = \frac{1}{2} \dim \Gorbit_\lambda 
  \\ &= 
  \begin{cases} \displaystyle
    \frac{1}{2} \Bigl(
      \frac{1}{2} n(2n-1) 
      - \frac{1}{2} \sum_i (\transpose{\lambda}_i)^2 
      + \frac{1}{2} \sum_{\text{$i$:odd}} m_i
    \Bigr)
    & (\frakg = \frako_n),
    \\ \displaystyle
    \frac{1}{2} \Bigl(
      \frac{1}{2} n(2n+1) 
      - \frac{1}{2} \sum_i (\transpose{\lambda}_i)^2 
      - \frac{1}{2} \sum_{\text{$i$:odd}} m_i
    \Bigr)
    & (\frakg=\fraksp_n , \; n : \text{even}),
  \end{cases}
\end{align*}
where $\transpose{\lambda} = (\transpose{\lambda}_1, \transpose{\lambda}_2, \ldots)$
is the transposed partition  of $\lambda$,
and $m_i$ denotes the multiplicity of $i$ in $\lambda$.
\end{lemma}

\begin{proof}
See \cite[Corollary~6.1.4]{Collingwood.McGovern.1993}, for example.
\end{proof}

\begin{lemma}
\label{lem:BDI-CI-CII-DIII-codim-1-inclusion}
The closure relations of nilpotent $K$-orbits on $\fraks$
of codimension one are given as follows.
\begin{rbenumerate}
\item 
For types {\upshape{}BDI} and {\upshape{}CI},
all the cover relations in Table~\ref{table:Ohta-Table-V-2}
are of codimension one
if and only if $T$ has no rows of length
between the longer length in $\overline{T}$
and the shorter length in $\overline{T}$ $($exclusive$)$.

\item 
For types {\upshape{}CII} and {\upshape{}DIII},
the closure relations of codimension one are
Case $(1)$ of {\upshape{}CII} and Case $(1)$ of {\upshape{}DIII} in Table~\ref{table:Ohta-Table-V-1}
such that $T$ has no rows of length
between the longer length in $\overline{T}$
and the shorter length in $\overline{T}$ $($exclusive$)$.
\end{rbenumerate}
\end{lemma}

\begin{lemma}
\label{lem:BDI-CI-CII-DIII-adjacent-relation}
For symmetric pairs $(G,K)$ of types {\upshape{}X} $=$ {\upshape{}BDI}, {\upshape{}CI}, {\upshape{}CII}, {\upshape{}DIII},
two nilpotent $K$-orbits
$\Korbit_T$ and $\Korbit_{T'}$
$(T, T' \in \sydx(\lambda; p,q))$
are adjacent in codimension one
if and only if $\overline{T}$ and $\overline{T'}$ are of the following form,
and $T$ has no rows of length
between the longer length in $\overline{T}$ 
and the shorter length in $\overline{T}$ $($exclusive$)$.

\bgroup\tiny
\noindent\hfil
$\begin{array}{lllllllll}
  \quad \mathrm{BDI}
  \\ \hline
  \quad \overline{T} & ~\qquad \overline{T'} 
  \\ \hline
  \begin{array}{l} \baDDab[2u+1] \\\abDba[2v+1] \end{array} & 
  \begin{array}{l} \abDDba[2u+1] \\\baDab[2v+1] \end{array}
  (u > v \ge 0)
\end{array}$ 
\quad
$\begin{array}{lllllllll}
  \quad \mathrm{CI}
  \\ \hline
  \quad \overline{T} & ~\qquad \overline{T'} 
  \\ \hline
  \begin{array}{l} \baDDba[2u] \\\abDab[2v] \end{array} & 
  \begin{array}{l} \abDDab[2u] \\\baDba[2v] \end{array}
  (u > v \ge 0)
\end{array}$
\bigskip \\ \hfil
$\begin{array}{lllllllll}
  \quad \mathrm{CII}
  \\ \hline
  \quad \overline{T} & ~\qquad \overline{T'} 
  \\ \hline
  \begin{array}{l} \abDDba[2u+1]\\\abDDba\\\baDab\\\baDab[2v+1] \end{array} & 
  \begin{array}{l} \baDDab[2u+1]\\\baDDab\\\abDba\\\abDba[2v+1] \end{array}
  (u > v \ge 0)
\end{array}$
\quad
$\begin{array}{lllllllll}
  \quad \mathrm{DIII}
  \\ \hline
  \quad \overline{T} & ~\qquad \overline{T'} 
  \\ \hline
  \begin{array}{l} \abDDab[2u]  \\\abDDab\\\baDba\\\baDba[2v] \end{array} & 
  \begin{array}{l} \baDDba[2u]  \\\baDDba\\\abDab\\\abDab[2v] \end{array}
  (u > v \ge 0)
\end{array}$
\egroup
\end{lemma}

\begin{proof}
For CII and DIII the assertion easily follows from 
Lemma~\ref{lem:BDI-CI-CII-DIII-codim-1-inclusion} (2).
For BDI and CI the assertion also follows from 
Lemma~\ref{lem:BDI-CI-CII-DIII-codim-1-inclusion} (1),
although we should mention that when $u=v+1$ we use 
Case (1) of Table~\ref{table:Ohta-Table-V-2},
and when where $u>v+1$ we use 
Cases (3) and (8) of Table~\ref{table:Ohta-Table-V-2}.
\end{proof}

From Lemma~\ref{lem:BDI-CI-CII-DIII-adjacent-relation}
we finally obtain the structure of the orbit graph.

\begin{theorem}[Description of orbit graph]
Let ${\rm X} = {\rm BDI}, {\rm CI}, {\rm CII}$ or ${\rm DIII}$.
Let $\lambda \in \yd_{\rm X}(n)$ be a partition, 
where $ n = p + q $ is the size of the matrix group $ G $ given in Table~\ref{table:ssp.in.this.paper} 
$($$p=q$ if ${\rm X} = {\rm CI}$ or ${\rm DIII};$ 
  $p$ and $q$ are even if ${\rm X} = {\rm CI}$$)$.
Then the orbit graph  $\Ograph_K(\Gorbit_\lambda)$ is described as follows.
The vertices are
\begin{align*}
  \{ \Korbit_T \mid T \in \sydx(\lambda;p,q) \},
\end{align*}
and for two vertices $\Korbit_T$ and $\Korbit_{T'}$,
there is an edge if and only if $\pi(T) - \pi(T')$ belongs to
\begin{align*}
  & \{ \pm(e_r - e_{r+1}) \mid 1 \le r \le k-1 \} \cup \{ \pm e_k \} & &
  \text{when ${\rm X} = {\rm BDI}, {\rm CI}$},
  \\
  & \{ \pm 2(e_r - e_{r+1}) \mid 1 \le r \le k-1 \} \cup \{ \pm 2e_k \} & & 
  \text{when ${\rm X} = {\rm CII}, {\rm DIII}$},
\end{align*}
where $\pi: \sydx(\lambda;p,q) \to \RR^k$ is the composite of
the natural inclusion $\sydx(\lambda;p,q) \to \syd(\lambda;p,q)$
and $\pi: \syd(\lambda;p,q) \to \RR^k$
defined in \eqref{eq:definition.of.pi}.
\end{theorem}

\begin{remark}
The set of possible edges is a proper subset of the set of vectors listed in the above theorem.  
In fact, possible edges are 
$\pm(e_r - e_{r+1})$ with $i_r$ and $i_{r+1}$ odd for ${\rm X} = {\rm BDI}$, 
and $\pm(e_r - e_{r+1})$ with $i_r$ and $i_{r+1}$ even
together with $\pm e_k$ if $i_k$ is even for ${\rm X} = {\rm CI}$.
For ${\rm X} = {\rm CII}$ and ${\rm DIII}$,
possible edges are twice the possible edges of BDI and CI,
respectively.
\end{remark}

\subsection{Induction of subgraphs}
We use the operation of removing two successive columns of the same length,
and the induction $\gind$ of graphs,
which are introduced in Subsection~\ref{subsec:AIII-induction.of.subgraph}.
It is easy to see that these two operations
preserve the type X
(${\rm X} = {\rm BDI}, {\rm CI}, {\rm CII}, {\rm DIII}$)
of the signed Young diagrams.
\begin{equation*}
\xymatrix @R=1pt {
\syd(\lambda';p',q') \ar[r]^{\gind} & \syd(\lambda; p, q) \\
\rotatebox{90}{$ \subset $} & \rotatebox{90}{$ \subset $} \\
\sydx(\lambda';p',q') \ar[r]^{\gind} & \sydx(\lambda; p, q) \\
}
\end{equation*}
We can conclude the following two lemmas 
by similar argument as in the case of AIII.

\begin{lemma}
\label{lem:number-of-connected-components-coincide-BDI-CI-CII-DIII}
Let ${\rm X} = {\rm BDI}, {\rm CI}, {\rm CII}$ or ${\rm DIII}$,
and $\lambda \in \yd_{\rm X}(n)$.
Let $\lambda' \in \yd_{\rm X}(n-2h)$ be the Young diagram
obtained by removing successive two columns of the same height $h$.
Then the number of connected components of $\Ograph_K(\Gorbit_\lambda)$
coincides with that of $\Ograph_{K'}(\Gporbit_{\lambda'})$,
where $(G,K)$ and $(G',K')$ are as follows.
\begin{gather*}
  \begin{array}{llllll}
    & (G,K) & (G',K')
    \\ \hline
    \mathrm{BDI} & (O_{p+q},O_p \times O_q) & (O_{p+q-2h},O_{p-h} \times O_{q-h}) 
    \\ 
    \mathrm{CI} & (Sp_{2p},GL_{p}) & (Sp_{2p-2h},GL_{p-h}) 
    \\ 
    \mathrm{CII} & (Sp_{p+q},Sp_{p}\times Sp_{q}) & (Sp_{p+q-2h},Sp_{p-h}\times Sp_{q-h}) 
    \\ 
    \mathrm{DIII} & (O_{2p},GL_{p}) & (O_{2p-2h},GL_{p-h}) 
  \end{array}
\end{gather*}
\end{lemma}

In the above lemma, 
in the case of type CII or DIII, 
the length $ h $ of the removed columns is always even 
since all parts of $ \lambda $ occur with even multiplicity 
(see \S~\ref{subsection:number.of.nilpotent.orbits.types.BDI-DIII} and \cite[Proposition 2.2]{Trapa.2005}).

\begin{lemma}
Let $\lambda$ and $\lambda'$ be as above, 
and we use the induction $\gind$.
If $Z'$ is a connected component of $\Ograph_{K'}(\Gporbit_{\lambda'})$,
then $\gind(Z')$ is a connected component of $\Ograph_{K}(\Gorbit_{\lambda})$.
This correspondence establishes a bijection between
the connected components of $\Ograph_{K'}(\Gporbit_{\lambda'})$
and those of $\Ograph_{K}(\Gorbit_{\lambda})$.
\end{lemma}

\subsection{Number of connected components}
To answer Problem~\ref{prob:general-case}
(\ref{prob:item:number.of.connected.components.of.graph}),
as in the case of AIII,
it suffices to consider the Young diagram with columns of different lengths
thanks to
Lemma~\ref{lem:number-of-connected-components-coincide-BDI-CI-CII-DIII}.

\begin{theorem}
\label{thm:number.of.connected.components-BDI-CI-CII-DIII}
Let ${\rm X} = {\rm BDI}, {\rm CI}, {\rm CII}$ or ${\rm DIII}$.
Let $\lambda \in \yd_{\rm X}(n)$,
and $n = p+q$ 
$($we put $p=q$ if ${\rm X} = {\rm CI}$ or ${\rm DIII}; $
  $p$ and $q$ are even if ${\rm X} = {\rm CII}$$)$.
\begin{rbenumerate}
\item
\label{thm:number.of.connected.components-BDI-CI-CII-DIII:1}
The orbit graph $ \Ograph_K(\Gorbit_{\lambda}) $ consists of a single vertex if and only if 
\begin{itemize}
  \item 
  For ${\rm X} = {\rm BDI}, {\rm CII}$; 
  the number of odd parts in $ \lambda $ is equal to $ p - q $, or 
  odd parts in $ \lambda $ have the same length.
  
  \item 
  For ${\rm X} = {\rm CI}, {\rm DIII}$; 
  the parts in $ \lambda $ are all odd.
\end{itemize}
\item
\label{thm:number.of.connected.components-BDI-CI-CII-DIII:2}
Let us assume that there are at least two vertices in $ \Ograph_K(\Gorbit_{\lambda}) $.  
Then the orbit graph $\Ograph_K(\Gorbit_\lambda)$ has no edges if and only if 
\begin{itemize}
  \item 
  each column length $h$ of $\lambda$ occurs odd times,
  or occurs even times and $\lambda_h$ is even,
  when ${\rm X} = {\rm BDI}, {\rm CII}$.
  \item 
  each column length $h$ of $\lambda$ occurs odd times,
  or occurs even times and $\lambda_h$ is odd,
  when ${\rm X} = {\rm CI}, {\rm DIII}$.
\end{itemize}
In particular,
if $\lambda$ has distinct column lengths,
then $\Ograph_K(\Gorbit_\lambda)$ has no edge.
\item
\label{thm:number.of.connected.components-BDI-CI-CII-DIII:3}
Assume that $\lambda$ has distinct column lengths.
In this case the number of the connected components
$($i.e., the number of vertices$)$ of the orbit graph $\Ograph_K(\Gorbit_\lambda)$
is given by
\begin{align}
\label{eq:number-of-conn-comp-BDI}
& \prod_{\substack{ 1\le s\le m\\ \text{$\lambda_{k_s} : $ odd}}} (1+t+\cdots+t^{k_s-k_{s-1}}) 
\Big|_{t^d}
&& ({\rm X} = {\rm BDI}),
\\
\label{eq:number-of-conn-comp-CI}
& \prod_{\substack{ 1\le s\le m\\ \text{$\lambda_{k_s} : $ even}}} (1 + k_s - k_{s-1})
&& ({\rm X} = {\rm CI}),
\\
\label{eq:number-of-conn-comp-CII}
& \prod_{\substack{ 1\le s\le m\\ \text{$\lambda_{k_s} : $ odd}}} (1+t+\cdots+t^{(k_s-k_{s-1})/2}) 
\Big|_{t^{d/2}}
&& ({\rm X} = {\rm CII}),
\\
\label{eq:number-of-conn-comp-DIII}
& \prod_{\substack{ 1\le s\le m\\ \text{$\lambda_{k_s} : $ even}}} (1 + \frac{k_s - k_{s-1}}{2})
&& ({\rm X} = {\rm DIII}),
\end{align}
where $ k_0 = 0 < k_1 < k_2 < \cdots < k_m$ are the {\upshape(}distinct{\upshape)} column lengths
of $\lambda$,
$f(t) \big|_{t^d}$ denotes the coefficient of $t^d$,
and $d$ is the number given by 
\begin{equation}
d := \frac{p-q+\#(\text{odd parts of $ \lambda $})}{2} .
\end{equation}
The number $d$ is always an integer 
if $\sydx(\lambda;p,q)$ is non-empty,
and is an even integer if $\sydcii(\lambda;p,q)$ is non-empty.
\end{rbenumerate}
\end{theorem}

\begin{proof}
\eqref{thm:number.of.connected.components-BDI-CI-CII-DIII:1}
By the description of primitives of the signed Young diagrams we get the desired conditions.  
Note that there is a unique filling of $ {+}/{-} $ signs for a pair of even (respectively odd) parts in the case of 
${\rm X} = {\rm BDI} $ or $ {\rm CII}$ 
(respectively ${\rm X} = {\rm CI} $ or $ {\rm DIII}$).

\eqref{thm:number.of.connected.components-BDI-CI-CII-DIII:2}
By the condition for two $K$-orbits to be adjacent in codimension one
(Lemma~\ref{lem:BDI-CI-CII-DIII-adjacent-relation}),
we immediately have the assertion.

\eqref{thm:number.of.connected.components-BDI-CI-CII-DIII:3}
Using the forms of primitives in Table~\ref{table:primitives-of-syd}
together with the condition on the number of signs,
we have the desired formula as in the case of AIII.
We omit the details.
\end{proof}

From this theorem the condition for an orbit graph to be connected is immediate.

\begin{corollary}
\label{corollary:connectedness.of.orbit.graph.type.BCD}
Under the same notation as in 
Theorem~\ref{thm:number.of.connected.components-BDI-CI-CII-DIII},
an orbit graph $\Ograph_K(\Gorbit_\lambda)$ is connected if and only if
\begin{itemize}
  \item {\upshape{}(BDI, CII)}
  there exists $ 0 \leq r \le s \leq \ell = \ell(\lambda) $ such that 
\begin{equation}  
    \label{eq:condition-for-connectedness-BDI-CII}
\begin{aligned} 
    &\text{
      $\lambda_1, \ldots, \lambda_r$ are even,
      $\lambda_{r+1}, \ldots, \lambda_s$ are odd,}\\
     &\quad\text{and $\lambda_{s+1}, \ldots, \lambda_{\ell}$ are even, }
  \end{aligned}
\end{equation}
or the number of odd parts in $ \lambda $ coincides with $ |p - q| $.

  \item {\upshape{}(CI, DIII)}
  there exists $ 0 \leq r \leq \ell = \ell(\lambda) $ such that 
  \begin{align} 
    \label{eq:condition-for-connectedness-CI-DIII}
    \text{
      $\lambda_1, \ldots, \lambda_r$ are odd,
      and $\lambda_{r+1}, \ldots, \lambda_{\ell}$ are even. }
  \end{align}
\end{itemize}
In particular, if $ \Gorbit_{\lambda} $ is an even nilpotent orbit, its orbit graph is connected.
\end{corollary}

\begin{proof}
We give the proof for ${\rm X}={\rm CI}$ and ${\rm BDI}$,
and the proofs are similar when ${\rm X} = {\rm CII}$ and ${\rm DIII}$.

Suppose that ${\rm X}={\rm CI}$, and $\lambda$ has distinct column lengths.
Equation (\ref{eq:number-of-conn-comp-CI}) is equal to one
if and only if the product is empty.
This means that $m \le 1$, and $ \lambda $ has at most one column.

For any $\lambda \in \yd_{\rm CI}(n)$,
we use the operation of removing successive two columns of the same length.
By repeating this operation $\lambda$ is reduced to a diagram
with distinct column lengths,
and the numbers of connected components of the corresponding orbit graphs
are equal.
Diagrams $\lambda \in \yd_{\rm CI}(n)$ which are reduced to diagrams
with at most one column are
of the form~(\ref{eq:condition-for-connectedness-CI-DIII}).

Next suppose that ${\rm X}={\rm BDI}$,
and $\lambda$ has distinct column lengths.
Equation (\ref{eq:number-of-conn-comp-BDI}) is equal to one
if and only if the product has at most one factor,
or $d = 0$ or equal to the highest degree of the polynomial (\ref{eq:number-of-conn-comp-BDI}), 
namely the number of odd parts of $\lambda$.  

The first condition is equivalent to $m \le 2$, 
namely, $ \lambda $ has at most two columns.
We again use the operation of removing columns,
and diagrams $\lambda \in \yd_{\rm BDI}(n)$ which are reduced to diagrams
with at most two columns are 
of the form~(\ref{eq:condition-for-connectedness-BDI-CII}).
The second condition is equivalent to 
$ \#(\text{odd parts of $ \lambda $}) = |p - q| $.  

Thus we obtain the desired condition.
\end{proof}

\section{Associated varieties of Harish-Chandra modules}
\label{section:AV.of.HC.module}

Let us consider 
Problems~\ref{prob:general-case}~\eqref{prob:item:irrep.for.connected.G.orbit} and 
\eqref{prob:item:irrep.for.connected.component} 
in this section for the symmetric pair of type AIII.

We write $ GL_n = GL_n(\CC) $, and put 
\begin{align*}
G &= GL_{p+q} = GL_n \;\; (n = p + q), &
K &= GL_p \times GL_q, \\
\frakg &= \frakgl_{p+q} = \frakgl_n, &
\frakk &= \frakgl_{p} \oplus \frakgl_{q}, \\
\fraks &= \Mat(p, q; \CC) \oplus \Mat(q, p; \CC).
\end{align*}
We consider a real form $ G_{\RR} = U(p, q) $ of $ G $, an indefinite unitary group of signature $ (p, q) $, and 
$ K_{\RR} = U(p) \times U(q) $ a maximal compact subgroup.  
Then $ (G,K) $ is the complexification of the Riemannian symmetric pair $ (G_{\RR}, K_{\RR}) $.  
Roughly saying, finitely generated admissible representations of $ G_{\RR} $ can be understood 
once we know completely about Harish-Chandra $ (\frakg, K) $-modules.  

The main subject in this section is a Harish-Chandra $ (\frakg, K) $-module $ X $ and 
its associated graph $ \AVgraph(X) $ (see Introduction for definition).  
The goal is the following theorem.

\begin{theorem}
\label{thm:for-AIII}
Consider the symmetric pair $ (G, K) = (GL_n, GL_p \times GL_q ) \;\; (n = p + q) $ 
associated with $ G_{\RR} = U(p,q) $.  
Let $ \Gorbit $ be a nilpotent $ G $-orbit in $ \frakg $.
\begin{rbenumerate}
\item
If the orbit graph $\Ograph_K(\Gorbit)$ is connected,
then there exists an irreducible Harish-Chandra $ (\frakg, K) $-module $X$ which satisfies 
$\Ograph_K(\Gorbit) = \AVgraph(X)$.  
Namely, the associated variety is $\AV(X) = \closure{\Gorbit \cap \fraks} $ for this Harish-Chandra module.

\item
More generally,
for any connected component $Z \subset \Ograph_K(\Gorbit)$,
there exists an irreducible Harish-Chandra $ (\frakg, K) $-module $X$
such that $Z = \AVgraph(X)$.  
\end{rbenumerate}

\smallskip

In Case {\upshape(1)}, we can choose 
$ X $ as an irreducible degenerate principal series representation, and 
in Case {\upshape(2)}, 
$ X $ can be chosen as a parabolic induction from a certain derived functor module, 
which we will describe explicitly below.
\end{theorem}

\begin{remark}
\label{remark:even.nilpotent.is.associated.variety}
(1)\ 
For an even nilpotent orbit $\Gorbit$, the orbit graph $ \Ograph_K(\Gorbit) $ is connected 
(Corollaries~\ref{cor:even.nilpotent.AIII} and \ref{corollary:connectedness.of.orbit.graph.type.BCD}), 
and the claim (1) of Theorem~\ref{thm:for-AIII} holds by Theorem 4.2 and Corollary 4.4 in \cite{Nishiyama.2011}.    
The associated Harish-Chandra module constructed in \cite{Nishiyama.2011} is also a degenerate principal series representation 
and it gives essentially the same representation as the present construction.  
See also \cite{Barbasch.Bozicevic.1999} and \cite{Matumoto.Trapa.2007}.

(2)\ 
In general, the containment 
$ \closure{\Gorbit \cap \fraks} \subset \closure{\Gorbit} \cap \fraks $ 
is strict even for an even nilpotent orbit $ \Gorbit $.  
See \cite[Remark~4.3]{Nishiyama.2011}.
\end{remark}

In the rest of this section, 
we prove Theorem~\ref{thm:for-AIII}.  
The proof is divided into several subsections.

\subsection{}
\label{subsection:one.step.induction}

Let us recall $ \lambda $ and $ \lambda' $ in \S \ref{subsec:AIII-induction.of.subgraph}.  
We mainly keep the notation in \S \ref{subsec:AIII-induction.of.subgraph} in this subsection.  
Thus we remove two columns of the same length $ h $ from $ \lambda $ and obtain $ \lambda' $.  
Put
\begin{equation*}
n' = n - 2 h , \quad 
(p', q') = ( p - h , q - h)
\end{equation*}
as before.  
Let us consider a real parabolic subgroup $ \psgR $ of $ \GR = U(p, q) $, whose Levi part is 
\begin{equation*}
\LR \simeq U(p', q') \times GL_h(\CC)  .
\end{equation*}
We realize $ \psgR $ in the following way.  
Let $ \{ e_i \mid 1 \leq i \leq n \} \subset \CC^n $ be the standard basis of $ \CC^n $, 
and we denote an indefinite Hermitian form $ (, ) $ by 
\begin{equation*}
(u, v) = \transpose{\conjugate{u}} I_{p,q} v 
\quad 
(u, v \in \CC^n) , \qquad
\text{ where } 
I_{p, q} = \begin{pmatrix} 1_p & \\ & - 1_q \end{pmatrix} .
\end{equation*}
Then $ \GR $ is realized as a matrix group which preserves the Hermitian form $ ( \, , \, ) $:
\begin{equation*}
\GR = U(p, q) = 
\{ g \in GL_n(\CC) \mid \transpose{\conjugate{g}} I_{p,q} g = I_{p,q} \} .
\end{equation*}
It is easy to see that a subspace 
$ V_h^{\pm} = \langle e_i \pm e_{n - i + 1} \mid 1 \leq i \leq h \rangle $ 
is totally isotropic with respect to $ ( , ) $.  
Then the parabolic subgroup 
\begin{equation*}
\psgR = \{ g \in U(p, q) \mid g(V_h^+) = V_h^+ \}
\end{equation*}
satisfies our requirement.  
In fact a Levi subgroup $ \LR $ is given by 
\begin{equation*}
\LR = \{ g \in U(p, q) \mid g(V_h^\pm) = V_h^\pm \} .
\end{equation*}
If we put 
$ W_{p', q'} = ( V_h^+ \oplus V_h^- )^{\bot} $, 
the orthogonal complement of $ V_h^+ \oplus V_h^- $ with respect to $ (, ) $, 
then $ \LR $ clearly preserves $ W_{p', q'} $ and 
\begin{equation*}
\LR \ni g \mapsto 
(g \restrict_{W_{p', q'}} , g \restrict_{V_h^+} ) \in U(p', q') \times GL_h(\CC) 
\end{equation*}
gives an isomorphism.  
Note that the Hermitian form $ (, ) $ restricted to $ W_{p', q'} $ has the signature 
$ (p', q') $.  

For $ \nu \in \CC $ and a (possibly infinite dimensional) admissible representation $ \pi' $ 
of $ \GpR = U(p', q') $, 
let $ \pi'(\nu) $ be an admissible representation of $ \LR $ defined by 
\begin{equation*}
\pi'(\nu)(g) = \bigl|\det(g \restrict_{V_h^+})\bigr|^{\nu} \, \pi'(g \restrict_{W_{p', q'}}) .
\end{equation*}
We extend it to $ \psgR $ in such a way that $ \pi'(\nu) $ is trivial on the unipotent radical, and denote it by the same notation $ \pi'(\nu) $.  
We define 
\begin{equation*}
I(\pi'; \nu) = I_{\psgR}^{\GR}(\pi'; \nu) := \Ind_{\psgR}^{\GR} \pi'(\nu);
\end{equation*}
here induction is normalized as in \cite[Chapter VII]{Knapp.1986}.
Assume that $ \pi'$ is an irreducible representation of $ \GpR $ and 
the associated variety of its primitive ideal is $ \closure{\Gporbit_{\lambda'}} $.  

\begin{lemma}
\label{key-lemma:induction.one.step}
For a generic $ \nu \in \CC $, 
the standard module $ I_{\psgR}^{\GR}(\pi'; \nu) $ is irreducible, 
and we have 
\begin{equation*}
\AVgraph(I_{\psgR}^{\GR}(\pi'; \nu)) 
= \gind_{(G', K')}^{(G, K)}( \AVgraph(\pi') ).  
\end{equation*}
In particular, 
if $ \AVgraph(\pi') $ is a connected graph, 
$ \AVgraph(I_{\psgR}^{\GR}(\pi'; \nu)) $ is also connected.
\end{lemma}

\begin{proof}
The irreducibility statement is well-known (e.g~\cite[Remark 1, page
  174]{Knapp.1986}).  For the remainder, we sketch two proofs.  
The first is essentially analytic (but uses the difficult results of
\cite{Schmid.Vilonen.2000} to pass from the analytic invariant
of wave front set to associated varieties).  The second is essentially
algebraic (but uses the difficult results of \cite[Chapter XI]{Knapp.Vogan.1995}
to rewrite  parabolically induced representations as cohomologically
induced instead).

For the first sketch we begin with a few generalities.  (The results
of the next two paragraphs hold in the generality of any real
reductive group $G_{\RR}$.)  Let $\caln(\frakg_{\RR})$ denote the
nilpotent cone of $\frakg_{\RR}$.  Given a finite-length
representation $\pi$ of $G_{\RR}$ on a Hilbert space, we let
$\WF(\pi) \subset \caln(\frakg_{\RR})$ denote its wave front set in
the sense of Howe \cite{Howe.1981}.  (The wave front set is most
naturally defined as a subset of $\frakg_{\RR}^{*}$; here and
elsewhere we identify $\frakg_\RR$ with $\frakg_\RR^{*}$ by means of
an invariant form.  Since $\WF(\pi)$ and the other invariants we
consider are invariant under scaling, the choice of form does
not matter.)  According to \cite[Theorem C]{Rossmann.1995},
$\WF(\pi)$ coincides with 
the asymptotic support 
$\AS(\pi)$ 
defined by Barbasch-Vogan
\cite{Barbasch.Vogan.1980}.  Moreover, \cite[Theorem D]{Rossmann.1995}
implies that if $\pi$ is assumed to be irreducible, then there are
$G_{\RR}$ orbits $\Rorbit_1, \dots, \Rorbit_\ell$ (each of which
generate the same $G$-orbit $\Gorbit$) such that
\begin{displaymath}
\WF(\pi) = \bigcup_{i=1}^{\ell} \closure{\Rorbit_i}.
\end{displaymath}
Finally, write $\Korbit_i$ for the $K$ orbit corresponding to $\Rorbit_{i}$
via the Seki\-guchi correspondence (e.g.~\cite[Chapter 9]{Collingwood.McGovern.1993}).
Then Schmid and Vilonen \cite{Schmid.Vilonen.2000} prove that
\begin{equation}
\label{e:BVconjecture}
\AV(\pi) = \bigcup_{i=1}^{\ell} \closure{\Korbit_i}.
\end{equation}

Now suppose $\pi$ is of the form $\Ind_{\psgR}^{\GR} (\pi')$
for an irreducible admissible representation $\pi'$ of $L_{\RR}$.
Using the inclusion $\lie{l}_{\RR}$ into $\frakg_{\RR}$, 
regard $\WF(\pi')$ as a subset of $\frakg_{\RR}$. 
We claim
\begin{equation}
\label{e:WFind}
\WF(\pi) = G_{\RR} \cdot \left (\WF(\pi') + \frakn_{\RR}\right ),
\end{equation}
where $\frakn_{\RR}$ denotes the Lie algebra of the nilradical $N_{\RR}$
of $P_{\RR}$.
According to \cite[Theorem C]{Rossmann.1995} mentioned above,
the assertion is equivalent to
\begin{equation}
\label{e:ASind}
\AS(\pi) = G_{\RR} \cdot \left ( \AS(\pi') + \frakn_{\RR} \right ).
\end{equation}
Under a technical positivity hypothesis
stated two sentences after \linebreak 
\cite[Equation (3)]{Barbasch.Vogan.1980},
\eqref{e:ASind} is proved in \cite[Theorem 3.5]{Barbasch.Vogan.1980}.
The positivity hypothesis may be verified as follows.
The construction
of \cite{Barbasch.Vogan.1980} assigns a real number to each
irreducible component of $\AS(\pi)$.  (Because there is
no need to normalize measures carefully in \cite{Barbasch.Vogan.1980}, these
real numbers are defined only up to positive scaling.  
Hence only their signs are well-defined in \cite{Barbasch.Vogan.1980}.)  
The crucial positivity hypothesis is that all of these
numbers are positive.
Meanwhile, after normalizing measures carefully, Rossmann
interprets these real numbers in \cite[Theorems B-C]{Rossmann.1995}
as integrals over certain Lagrangian cycles.  The
cycles are made somewhat more explicit in the work of Schmid-Vilonen
\cite{Schmid.Vilonen.1998}, where their positivity becomes apparent.  (In fact,
in \cite{Schmid.Vilonen.2000} they are shown to coincide with the positive integers
appearing in the associated cycle of $\pi$ (in the sense of
\cite{Vogan.1991}).)  Thus the positivity hypothesis always holds.  Hence
so does \eqref{e:ASind} (and equivalently \eqref{e:WFind}).
As mentioned around \eqref{e:BVconjecture}, 
the main results of \cite{Schmid.Vilonen.2000} then allows us to 
interpret the assertion of \eqref{e:WFind} as a computation of associated
varieties.

Return to the setting of the lemma. Suppose we
are given a nilpotent $K'$ orbit $\Kporbit$ parameterized by $T' \in
\syd(\lambda'; p',q')$.  Let $\Rporbit$ denote the nilpotent $G'_{\RR}$
orbit corresponding to $\Kporbit$ via the Sekiguchi correspondence.
Consider
\[
G_\RR\cdot \left ( \closure{\Rporbit} + \frakn_\RR \right).
\]
This is a closed $G_\RR$-invariant set of nilpotent elements 
(see \cite[Theorem 7.1.3]{Collingwood.McGovern.1993}).  Hence it may
be written as
\[
\closure{\Rorbit_1} \cup \cdots \cup \closure{\Rorbit_\ell}
\]
for nilpotent $G_\RR$ orbits $\Rorbit_i$, each of which is parameterized by
a signed Young diagram $T_i$ of signature $(p,q)$.  
By the discussion around \eqref{e:BVconjecture}
and \eqref{e:WFind}, 
the lemma amounts to establishing
$\{T_1, \dots, T_\ell\}$ is obtained from $T'$ by the 
procedure described
before Lemma \ref{lemma:gind.of.connected.components}.
This is a direct calculation whose details we will give in Appendix.

We next turn to the second approach to proving the lemma.
As remarked above, the main point is to rewrite $I_{\psgR}^{\GR}(\pi'; \nu)$
as a cohomologically induced representation.  To get started,
fix a $\theta$-stable parabolic subalgebra $\frakq = \frakl \oplus \fraku$
of $\frakg$ such that
\[
\frakl_\RR :=\frakl \cap \conjugate{ \frakl} \simeq \fraku(p',q') \oplus \fraku(h,h).
\]
Let $\LqR$ denote the analytic subgroup of $\GR$ with Lie algebra
$\frakl_\RR$, and set $S = \dim(\fraku \cap \frakk)$.  
Of course $\LqR \simeq \GR' \times \GR''$ with $\GR' =U(p',q')$ as above
and $\GR'' = U(h,h)$.
Fix a Cartan subalgebra $\frakh$ of $\frakl$ (and hence
of $\frakg$), write $\Delta(\fraku) \subset \frakh^*$ 
for the roots of $\frakh$ in $\fraku$, and let $\delta_\fraku$ denote
the half-sum of the elements of $\Delta(\fraku)$.
We follow the notation of
\cite[Chapter 5]{Knapp.Vogan.1995}.
In particular, 
given a $(\frakl, L^{\mathfrak{q}} \cap K)$ module $\pi' \boxtimes \pi''$, we may form the
cohomologically induced $(\frakg,K)$ module $\call_S(\pi' \boxtimes \pi'')$.

Next let $\psgR''$ denote a real parabolic subgroup of $\GR''$ with Levi
factor $GL(h,\CC)$.  Fix $\nu\in \CC$ as in the statement of the lemma 
and extend the character $|\det|^\nu$ of $GL(h,\CC)$ trivially
to the nilradical of $\psgR''$.  Set
\begin{equation}
\label{e:pi''}
\pi'' : = \Ind_{\psgR''}^{\GR''}(|\det|^\nu),
\end{equation}
where again the induction is normalized.  We may choose $\nu$ so that
\begin{enumerate}
\item[(a)] $\pi''$ is irreducible; and

\item[(b)] if $\chi \in \frakh^*$ denotes (a representative of) the
  infinitesimal character of $\pi' \boxtimes \pi''$, then
\[
\mathrm{Re}\left(\alpha^\vee(\chi + \delta_\fraku)\right) > 0
\quad \text{for all $\alpha \in \Delta(\fraku)$}.
\]
(In the terminology of \cite[Definition 0.49]{Knapp.Vogan.1995},
$\chi$ is said to be in the good range for $\frakq$.)
\end{enumerate}
With such a choice of $\nu$ fixed, our main technical claim is
as follows:
\begin{equation}
\label{e:transferclaim}
I_{\psgR}^{\GR}(\pi'; \nu) = \call_S(\pi' \boxtimes \pi'').
\end{equation}
To prove this, we compute the Langlands (quotient) parameters of both
sides. 
 For the left-hand side, we may assume the Langlands parameters
of $\pi'$ are given.  They consist of a cuspidal parabolic subgroup
$M_\RR'A_\RR'N_\RR'$ 
of $\GR'$, a discrete series or limit of discrete
series representations $\xi'$ of $M'_\RR$, and a suitably positive
character $\eta'$ of $A'_\RR$.  See for example the discussion at the
beginning of \cite[XI.9]{Knapp.Vogan.1995}.   In the notation of
Knapp-Vogan, let $I'(\xi',\eta')$ denote the standard continuous representation
of $\GR'$
parameterized by $M_\RR'A_\RR'N_\RR'$, $\xi'$, and $\eta'$.  By construction,
we have a surjection
\begin{equation}
\label{e:surjection'}
I'(\xi',\eta') \longrightarrow \pi'.
\end{equation}
Next we consider the
Langlands parameters for the character $|\det|^\nu$ of $GL(h,\CC)$.
Of course this is well-known (see, for example, \cite[Theorem
  4]{Knapp.1994}).  The parameters consist of a Borel subgroup $M^{GL}_\RR A^{GL}_\RR N^{GL}_\RR$
of $GL(h,\CC)$, 
a character $\xi^{GL}$ of $M^{GL}_\RR$, and an appropriately
positive character $\eta^{GL}$ of $A_\RR$.   In the obvious notation,
we have a surjection
\begin{equation}
\label{e:surjectionGL}
I^{GL}(\xi^{GL},\eta^{GL}) \longrightarrow |\det|^\nu.
\end{equation}
We can combine these two
Langlands parameters to get a Langlands parameter for $G_\RR$:
there exists a cuspidal parabolic subgroup $M_\RR A_\RR N_\RR$ of $\GR$
whose intersection with $\GR'$ is $M'_\RR A'_\RR N'_\RR$ and whose intersection
with $GL(h,\CC)$ is $M^{GL}_\RR A^{GL}_\RR N^{GL}_\RR$.  In this case 
\begin{align*}
M_\RR &= M'_\RR \times M^{GL}_\RR\\
A_\RR &= A'_\RR \times A^{GL}_\RR,
\end{align*}
and so we can form $\xi = \xi' \boxtimes \xi^{GL}$ and 
$\eta = \eta' \boxtimes \eta^{GL}$.  Then $N_\RR$ can be chosen so that
$(M_\RR A_\RR N_\RR, \xi, \eta)$
is a quotient Langlands parameter for $\GR$.  If we let
$I(\xi, \eta)$ denote the corresponding standard continuous representation
of $\GR$, then \eqref{e:surjection'}, \eqref{e:surjectionGL}, and
induction in stages give a surjection
\begin{equation}
\label{e:surjection}
I(\xi,\eta) \longrightarrow I_{\psgR}^{\GR}(\pi'; \nu),
\end{equation}
the image of which we have assumed is irreducible.
Thus the triplet $(M_\RR A_\RR N_\RR, \xi, \eta)$ is indeed a quotient Langlands parameter
for the induced representation $I_{\psgR}^{\GR}(\pi'; \nu)$, the left-hand side of \eqref{e:transferclaim}.

The more difficult part of the argument is computing the Langlands parameters
of right-hand side of \eqref{e:transferclaim}.  Fortunately 
\cite[Theorem 11.216]{Knapp.Vogan.1995} 
explains how to compute the
Langlands parameters of $\call_S(\pi' \boxtimes \pi'')$ in terms of
those for $\pi'$ and $\pi''$.  (To apply this theorem, 
we need the ``good range'' hypothesis detailed in item (b) above.)  We may assume,
as above, that we are
given the Langlands parameters of $\pi'$.  To compute the Langlands
parameters of $\pi''$, note that $M_\RR^{GL}A_\RR^{GL} \simeq (\CC^\times)^h$
is the Levi factor of a cuspidal parabolic subgroup $M''_\RR A''_\RR N''_\RR$
of $\GR''$.  Thus if we set $\xi'' = \xi^{GL}$ and $\eta'' = \eta^{GL}$,
we can choose $N''_\RR$ so that $(M''_\RR A''_\RR N''_\RR, \xi'', \eta'')$
is a Langlands parameter for $G_\RR''$.  Moreover, induction in stages
and \eqref{e:surjectionGL} imply that there is a surjection 
\[
I''(\xi'', \eta'') \longrightarrow \pi'',
\]
the image of which is irreducible by hypothesis.
So $(M''_\RR A''_\RR N''_\RR, \xi'', \eta'')$ is indeed a Langlands quotient
parameter for $\pi''$.  Then Theorem 
\cite[Theorem 11.216]{Knapp.Vogan.1995} gives that the Langlands parameters
of the right-hand side of \eqref{e:transferclaim} are exactly
those of the left-hand side computed above.
Thus \eqref{e:transferclaim} follows.

Given \eqref{e:transferclaim}, we can complete the second proof
of the lemma relatively easily.  The first ingredient is to apply a general
result about associated varieties of derived functor modules to the 
right-hand side of \eqref{e:transferclaim},
\begin{equation}
\label{e:avcomp}
\AV\left( \call_S(\pi' \boxtimes \pi'') \right ) 
= K \cdot \left( \AV(\pi' \boxtimes \pi'') + (\fraku \cap \fraks)\right).
\end{equation}
(In the case that $\pi' \boxtimes \pi''$ is one-dimensional, a
well-known argument is sketched in the introduction of
\cite{Trapa.2005}; the general case follows in much the same way.)
The next ingredient is the computation of $\AV(\pi'')$:
\cite[Corollary 5.4]{Nishiyama.2011} 
proves 
that $\AV(\pi'')$ consists
of the closures of the $h+1$ nilpotent orbits for $U(h,h)$ whose shape 
consists of $h$ rows of two boxes.  Finally, 
Proposition 3.1 and Corollary 3.2 of \cite{Trapa.2005} explain how to
compute the right-hand side of \eqref{e:avcomp} in terms of signed
Young diagrams.  Combined with \eqref{e:transferclaim}, the result
is that $\AV(I_{\psgR}^{\GR}(\pi'; \nu))$ consists of the closures of
the $K$ orbits parameterized by the signed Young diagrams obtained from
those parameterizing the irreducible components of
$\AV(\pi')$ by the 
procedure described
before Lemma \ref{lemma:gind.of.connected.components}.
This completes the second proof of the lemma.
\end{proof}

\subsection{}
\label{subsection:case.of.connected.orbit.graph}

Now let us consider the first claim of Theorem \ref{thm:for-AIII}, 
so we assume that $ \Ograph_K(\Gorbit_{\lambda}) $ is connected.  
Then, by Corollary \ref{cor:AIII-condition-of-connectedness}, 
if we remove two columns of the same length from $ \lambda $ repeatedly, 
finally we reach $ \lambda' $ of a diagram with only one column 
(possibly $ \lambda' $ is an empty diagram).  
Thus we have 
\begin{equation*}
\lambda = \lambda' + \sum_{i = 1}^t [2^{h_i}] ,
\end{equation*}
for some $ h_i $'s,
where $[2^{h_i}]$ denotes the partition $(2,2,\ldots,2)$ of length $h_i$.
Put 
\begin{equation*}
h = \sum_{i = 1}^t h_i , \quad 
n' = n - 2 h , \quad 
(p', q') = ( p - h , q - h) .  
\end{equation*}
Note that the notation $ \lambda', p', q' $ etc.{} is used in slightly different way 
from the former subsection.  
With these notations, we have $ \lambda' = (1^{n'}) \; (n' \geq 0) $ and 
if $ n' = 0 $, it means that $ \lambda' $ is an empty diagram.  

Let us consider a real parabolic subgroup $ \psgR $ of $ \GR = U(p, q) $ with Levi part 
\begin{equation}
\label{eq:LR.isomorphic.Upq.t-timesGL}
\LR \simeq U(p', q') \times GL_{h_1}(\CC) \times \cdots \times GL_{h_t}(\CC) .
\end{equation}
The construction of $ \psgR $ is similar to that in the former subsection.  
In fact, we simply repeat the procedure in \S \ref{subsection:one.step.induction} $ t $-times.  
Now consider a character $ \chi_{\nu} $ of $ \LR $ with a parameter 
$ \nu = ( \nu_1, \dots, \nu_t) \in \CC^t $ defined by 
\begin{equation*}
\chi_{\nu}(g) = \prod_{i = 1}^t \bigl| \det g_i \bigr|^{\nu_i}, 
\end{equation*}
where $ g_i \in GL_{h_i}(\CC) $ is the $ GL_{h_i} $-component of 
$ g \in \LR $  under the isomorphism \eqref{eq:LR.isomorphic.Upq.t-timesGL}.

Let $ I(\nu) = I_{\psgR}^{\GR}(\nu) $ be a degenerate principal series defined by 
\begin{equation*}
I(\nu) := \Ind_{\psgR}^{\GR} \chi_{\nu} ,
\end{equation*}
where $ \chi_{\nu} $ is extended to a character of $ \psgR $ which is trivial on the unipotent radical.  
Then we have

\begin{lemma}
For a generic $ \nu \in \CC^t $, 
the degenerate principal series 
$ I(\nu) = I_{\psgR}^{\GR}(\nu) $ is irreducible, and 
we have 
\begin{equation*}
\AVgraph(I(\nu)) = \Ograph_K(\Gorbit_{\lambda}).
\end{equation*}
\end{lemma}

\begin{proof}
We use Lemma \ref{key-lemma:induction.one.step} repeatedly, and conclude that 
\begin{equation*}
\AVgraph(I(\nu)) = \gind_{(G', K')}^{(G, K)}( \{ T' \} ),
\end{equation*}
where $ T' $ is the unique signed Young diagram in $ \syd(\lambda'; p', q') $ with a single column 
(possibly an empty diagram).  
By Lemma \ref{lemma:gind.of.connected.components}, we have 
$ \gind(\{ T' \}) = \Ograph_K(\Gorbit_{\lambda}) $ since the right hand side is a connected graph.
\end{proof}

\subsection{}

Let us consider a general partition $ \lambda $ of size $ n $.  
If we remove two columns of the same length successively from $ \lambda $, 
finally we obtain $ \lambda' $ with distinct column length.  
By Theorem \ref{thm:number.of.connected.components.of.orbit.graph}, 
$ \Ograph_{K'}(\Gporbit_{\lambda'}) $ is totally disconnected and individual 
$ \Kporbit_{T'} $ constitutes a connected component of the orbit graph.  
Thus, by Lemma \ref{lemma:gind.of.connected.components},
$ \gind(\{ \Kporbit_{T'} \}) \; (T' \in \syd(\lambda'; p', q')) $ 
exhausts connected components of 
$ \Ograph_K(\Gorbit_{\lambda}) $.  
By a result of Barbasch-Vogan \cite[Theorem 4.2]{Barbasch.Vogan.1983}, 
there exists a derived functor module $ \pi_{T'} $ of $ \GpR $ with the associated variety 
\begin{equation*}
\AV(\pi_{T'}) = \closure{\Kporbit_{T'}} .
\end{equation*}
We construct a real parabolic $ \psgR $ of $ \GR $ just as in 
the former subsection \S \ref{subsection:case.of.connected.orbit.graph}, 
and define an admissible representation $ \pi_{T'}(\nu) \; (\nu \in \CC^t) $ of the Levi subgroup 
$ \LR $ by 
\begin{equation*}
\pi_{T'}(\nu) (g) 
= \prod_{i = 1}^t \bigl| \det (g_i) \bigr|^{\nu_i} \cdot \pi_{T'}(g_0) , 
\end{equation*}
where
\begin{equation*}
g = ( g_0 ; g_1, \dots, g_t ) 
\in U(p', q') \times GL_{h_1}(\CC) \times \cdots \times GL_{h_t}(\CC) = \LR.
\end{equation*}
Let us consider a standard module 
\begin{equation*}
I(\pi_{T'}; \nu) = I_{\psgR}^{\GR}(\pi_{T'}; \nu) := \Ind_{\psgR}^{\GR} \pi_{T'}(\nu) .
\end{equation*}

\begin{lemma}
With the above notations, the associated graph 
\begin{equation*}
\AVgraph( I(\pi_{T'}; \nu) ) = \gind( \{ T' \} ) 
\end{equation*}
is a connected component of $ \Ograph_K(\Gorbit_{\lambda}) $, 
and these associated graphs exhaust connected components of the orbit graph 
$ \Ograph_K(\Gorbit_{\lambda}) $.
\end{lemma}

\begin{proof}
By repeated use of Lemma \ref{key-lemma:induction.one.step}, 
we obtain the first equality $ \AVgraph( I(\pi_{T'}; \nu) ) = \gind( \{ T' \} ) $.  
The right hand side is connected and it gives a bijection from 
$ \syd(\lambda'; p', q') $ to the connected components of the orbit graph 
$ \Ograph_K(\Gorbit_{\lambda}) $ by Lemma \ref{lemma:gind.of.connected.components}.
\end{proof}

\section{Appendix}

In this appendix, we calculate the induction of nilpotent orbits explicitly.  

\smallskip

Let $ \GR = U(p,q) $ and 
take a parabolic subgroup $ \psgR $ as in \S~\ref{subsection:one.step.induction}.  
The Levi part $ L_{\RR} $ of $ \psgR $ is isomorphic to $ GL_h(\CC) \times U(p', q') $, 
where $ (p', q') = (p - h, q -h) $.  
We take a nilpotent orbit $ \Rporbit $ of $ \GpR := U(p', q') $ and extend it trivially to $ L_{\RR} $, 
which we denote by the same notation.  
Then what we want to know is the induction 
$ \Ind_{\psgR}^{\GR} \Rporbit := $ the largest nilpotent orbit contained in $ G_\RR\cdot \left ( \closure{\Rporbit} + \frakn_\RR \right) $.
For this, we pick nilpotent elements $ X \in \Rporbit $ and 
$ \Xi \in \frakn_{\RR} $, and 
calculate the Jordan normal form of $ X + \Xi $.  
We assume that the Jordan normal form of $ X $ corresponds to a partition 
$ \lambda' = (\lambda'_i)_{1 \leq i \leq \ell} $ of $ p' + q' $.  
Then the Jordan normal form of $ X + \Xi $ corresponds to 
a partition which is obtained from $ \lambda' $ by adding $ 2 $ in different $ h $-places (and rearranging the parts in nonincreasing order).
We will explain this below, but some remarks are in order.

In fact, the nilpotent orbit $ \Rporbit $ is parametrized by a signed Young diagram $ T' $, 
and the so-obtained partition (adding $ 2 $ in $ h $-places) has unique signature compatible with $ T' $.  
So the Jordan normal form (the shape of the diagram) is enough to specify the obtained $ \GR $-orbits.  
Among them, the largest one with respect to the closure relation is 
$ \lambda' + [2^h] $ (adding $ 2 $ in the first $ h $-places).  
This is what we want in \S~\ref{subsection:one.step.induction}.  

So the rest of Appendix is devoted to specify $ X $ and $ \Xi $ in the matrix form and calculate 
the Jordan normal form of $ X + \Xi $ explicitly.

\medskip

We realize $ \GR = U(p, q) $ as in \S~\ref{subsection:one.step.induction}, 
and denote by $ e_i $ the fundamental vector with $ 1 $ in the $ i $-th coordinate and $ 0 $ elsewhere.  
We choose a basis of the indefinite unitary space $ \CC^{p, q} $ as 
\begin{equation*}
\{ e_i + e_{n - i +1} \}_{1 \leq i \leq h} \cup \{ e_{h + j} \}_{1 \leq j \leq p' + q'} \cup \{ e_i - e_{n - i +1} \}_{1 \leq i \leq h} .
\end{equation*}
Using the coordinate for this basis, 
an element in $ \lie{u}(p, q) $ is represented in the form
\begin{equation*}
\begin{pmatrix}
A & \alpha & \beta & B \\
\gamma & X_{11} & X_{12} & -\alpha^{\ast} \\
\delta & X_{21} & X_{22} & \beta^{\ast} \\
C & - \gamma^{\ast} & \delta^{\ast} & - A^{\ast}
\end{pmatrix}
\qquad
\begin{array}{l}
X = (X_{ij}) \in \lie{u}(p', q'), \; A \in \lie{gl}_h(\CC) 
\\
B^{\ast} = - B, \; C^{\ast} = - C
\end{array}
\end{equation*}
In this coordinate, an element in  
the Lie algebra $ \lie{p}_{\RR} $ of the parabolic subgroup $ \psgR $ is represented in the block upper triangular form
\begin{equation*}
\begin{pmatrix}
A & \xi & B \\
 & X & \eta \\
 & & - A^{\ast}
\end{pmatrix}
\qquad
\xi = ( \alpha, \beta), \; \eta = \begin{pmatrix} - \alpha^{\ast} \\ \beta^{\ast} \end{pmatrix} .
\end{equation*}
We denote an element in the nilpotent radical $ \frakn_{\RR} $ by 
\begin{equation}
\label{eqn:nilpotents.in.Xi}
\Xi = 
\begin{pmatrix}
0 & \xi & B \\
 & 0 & \eta \\
 & & 0
\end{pmatrix}
\end{equation}
and pick a nilpotent element $ X \in \lie{u}(p', q') $ corresponding to the signed Young diagram $ T' $ above, 
which has the shape $ \lambda' $.  
By abuse of notation, we denote the embedded $ X $ into $ \lie{g}_{\RR} = \lie{u}(p, q) $ by the same letter $ X $.  
The embedding is specified by the matrix form above.
Then we calculate 
\begin{equation*}
(X + \Xi)^k = 
\begin{pmatrix}
0 & \xi X^{k - 1} & \xi X^{k - 2} \eta \\
 & X^k & X^{k - 1} \eta \\
 & & 0
\end{pmatrix}
\qquad 
(k \geq 2).
\end{equation*}
From this, we conclude that if $ X $ is $ k $-step nilpotent, then $ X + \Xi $ is $ (k + 2) $-step nilpotent 
extending the length by $ 2 $.  

Let us try to get more specific information.  
We re-arrange (a part of) the basis 
$ e_{h + 1} , e_{h + 2}, \dots, e_{h + p' + q'} $ to get a Jordan normal form 
$ X = J_{\lambda_1'} \oplus J_{\lambda_2'} \oplus \cdots \oplus J_{\lambda_{\ell}'} $, where 
$ J_m $ denotes a Jordan cell of size $ m $ with zeroes on the diagonal and $ 1 $'s on the upper diagonal (and zero elsewhere).
The calculation tells that we can enlarge the Jordan cell by $ 2 $ in each cell.  
But since the rank of $ \xi $ (or $ B $) is at most $ h $, we can choose at most $ h $-cells freely.

We exhibit this by an example, where $ X $ has $ 3 $ cells and $ h = 2 $.
Also, let us put $ \lambda' = (m_1, m_2, m_3) $, 
so that $ X = J_{m_1} \oplus J_{m_2} \oplus J_{m_3} $.  
For this $ X $ we have a direct decomposition of the vector space 
$ V = \CC^{p', q'} $ into 
$ V = V(m_1) \oplus V(m_2) \oplus V(m_3) $ 
such that $ X $ acts on $ V(m_i) $ by $ J_{m_i} $.  
Choose a vector $ v_i \in V(m_i) $ so that $ X^{m_i -1} v_i = J_{m_i}^{m_i - 1} v_i \neq 0 $.  
Then $ (v_i, X^{m_i - 1} v_i) \neq 0 $, where $ ({-},{-}) $ is the indefinite Hermitian form on $ V $.  
For this, see \cite[\S~9.3]{Collingwood.McGovern.1993}.  

Put $ \eta_{ij} = (v_i, v_j) $.  Then Equation 
\eqref{eqn:nilpotents.in.Xi} for $ \eta = \eta_{ij} $ defines $ \xi = \xi_{ij} $.  
It is easy to see that 
if $ \transpose{\xi_{ij}} = (v_i^{\ast}, v_j^{\ast}) $, 
then $ \transpose{v_i^{\ast}} X^{k -2} v_j = \delta_{ij} (v_i, X^{k -2} v_j) $ and 
we get 
\begin{equation*}
\xi_{ij} X^{k -2} \eta_{ij} = 
\begin{pmatrix}
(v_i, X^{k -2} v_i) & 0 \\
0 & (v_j, X^{k -2} v_j) 
\end{pmatrix} 
.
\end{equation*}
Now denote by $ \Xi_{ij} \in \frakn_{\RR} $ the matrix $ \Xi $ in 
\eqref{eqn:nilpotents.in.Xi} replaced $ \eta = \eta_{ij} $ and $ \xi = \xi_{ij} $.  
Then the Jordan type of $ X + \Xi_{12} $ is 
$ J_{m_1 + 2} \oplus J_{m_2 + 2} \oplus J_{m_3} $.  
Other cases where the induced nilpotent is $ X + \Xi_{ij} $ can be treated similarly.


\begin{thebibliography}{KnV96}

\bibitem[BB99]{Barbasch.Bozicevic.1999}
Dan Barbasch and Mladen Bo{\v{z}}i{\v{c}}evi{\'c}, \emph{The associated variety
  of an induced representation}, Proc. Amer. Math. Soc. \textbf{127} (1999),
  no.~1, 279--288. \MR{1458862 (99b:22024)}

\bibitem[BC77]{Burgoyne.Cushman.1977}
N.~Burgoyne and R.~Cushman, \emph{Conjugacy classes in linear groups}, J.
  Algebra \textbf{44} (1977), no.~2, 339--362. \MR{0432778 (55 \#5761)}

\bibitem[BV80]{Barbasch.Vogan.1980}
Dan Barbasch and David Vogan, \emph{The local structure of characters}, 
J.~Func.~Anal.~\textbf{37} (1980), 27--55.
\MR{0576644 (82e:22024)} 

\bibitem[BV83]{Barbasch.Vogan.1983}
Dan Barbasch and David Vogan, 
\emph{Weyl group representations and nilpotent
  orbits}, Representation theory of reductive groups ({P}ark {C}ity, {U}tah,
  1982), 21--33, Progr. Math., vol.~40, Birkh\"auser Boston, Boston, MA, 1983.
  \MR{733804 (85g:22025)}

\bibitem[CM93]{Collingwood.McGovern.1993}
David~H. Collingwood and William~M. McGovern, \emph{Nilpotent orbits in
  semisimple {L}ie algebras}, Van Nostrand Reinhold Mathematics Series, Van
  Nostrand Reinhold Co., New York, 1993. \MR{1251060 (94j:17001)}

\bibitem[Djo82]{Djokovic.1982}
Dragomir~{\v{Z}}. Djokovi{\'c}, \emph{Closures of conjugacy classes in
  classical real linear {L}ie groups. {II}}, Trans. Amer. Math. Soc.
  \textbf{270} (1982), no.~1, 217--252. \MR{642339 (84f:22016)}

\bibitem[Hel78]{Helgason.1978}
Sigurdur Helgason, \emph{Differential geometry, {L}ie groups, and symmetric
  spaces}, Pure and Applied Mathematics, vol.~80, Academic Press Inc. [Harcourt
  Brace Jovanovich Publishers], New York, 1978. \MR{514561 (80k:53081)}

\bibitem[Ho79]{Howe.1981}
Roger Howe,
\emph{Wave front sets of representations of Lie groups},
Automorphic forms, representation theory and arithmetic (Bombay, 1979), 117--140, 
Tata Inst. Fund. Res. Studies in Math., vol.~10, 
Tata Inst. Fundamental Res., Bombay, 1981.
\MR{0633659 (83c:22014)}
 
\bibitem[K94]{Knapp.1994} A.~W.~Knapp,
\emph{Local Langlands correspondence: the Archimedean case},
Motives (Seattle, WA, 1991), 393--410, Proc.~Sympos.~Pure Math.,
vol.~{55}, AMS, Providence, RI, 1994.
\MR{1265560 (95d:11066)}


\bibitem[Kn86]{Knapp.1986} Anthony W.~Knapp, \emph{Representations
of Semisimple Lie Groups: An Overview Based on Examples}, 
Princeton Mathematical Series, vol.~36,
Princeton University Press, Princeton, NJ, 1986.
\MR{855239 (87j:22022)}



\bibitem[KnV96]{Knapp.Vogan.1995} Anthony W.~Knapp and David A.~Vogan, Jr.,
\emph{Cohomological Induction and Unitary Representations}, 
Princeton Mathematical Series,
vol.~ 45,
Princeton University Press, Princeton, NJ, 1995.
\MR{1330919 (96c:22023)}


\bibitem[KP79]{Kraft.Procesi.1979}
Hanspeter Kraft and Claudio Procesi, \emph{Closures of conjugacy classes of
  matrices are normal}, Invent. Math. \textbf{53} (1979), no.~3, 227--247.
  \MR{549399 (80m:14037)}

\bibitem[KR71]{Kostant.Rallis.1971}
B.~Kostant and S.~Rallis, \emph{Orbits and representations associated with
  symmetric spaces}, Amer. J. Math. \textbf{93} (1971), 753--809. \MR{0311837
  (47 \#399)}

\bibitem[MT07]{Matumoto.Trapa.2007}
Hisayosi Matumoto and Peter~E. Trapa, \emph{Derived functor modules arising as
  large irreducible constituents of degenerate principal series}, Compos. Math.
  \textbf{143} (2007), no.~1, 222--256. \MR{2295203 (2008f:22013)}

\bibitem[{Nis}11]{Nishiyama.2011}
K.~{Nishiyama}, \emph{{Asymptotic cone of semisimple orbits for symmetric
  pairs}}, Adv.~Math., \textbf{226} (2011), 4338--4351.

\bibitem[Oht86]{Ohta.1986}
Takuya Ohta, \emph{The singularities of the closures of nilpotent orbits in
  certain symmetric pairs}, Tohoku Math. J. (2) \textbf{38} (1986), no.~3,
  441--468. \MR{854462 (88k:32077)}

\bibitem[Oht91]{Ohta.1991}
Takuya Ohta, 
\emph{The closures of nilpotent orbits in the classical symmetric
  pairs and their singularities}, Tohoku Math. J. (2) \textbf{43} (1991),
  no.~2, 161--211. \MR{1104427 (93c:22036)}
  
\bibitem[Ro95]{Rossmann.1995}
Wolf Rossmann,
\emph{Picard-Lefschetz theory and characters of a semisimple Lie group}.
Invent.~Math., \textbf{121} (1995), 570--611. \MR{1353309 (96j:22017)}
 
 
\bibitem[SV98]{Schmid.Vilonen.1998} 
Wilfried Schmid and Kari Vilonen,
{\emph Two geometric character formulas for reductive Lie groups},
J.~Amer.~Math.~Soc., \textbf{11} (1998), no.~4, 799--867.
\MR{1612634 (2000g:22020)}

  
\bibitem[SV00]{Schmid.Vilonen.2000}
Wilfried Schmid and Kari Vilonen,
\emph{Characteristic cycles and wave front cycles of representations of
   reductive Lie groups}, Ann.~of Math.~(2) \textbf{151} (2000), 1071--1118.
\MR{1779564 (2001j:22017)}

  
\bibitem[Tr05]{Trapa.2005}
Peter E.~Trapa, 
\emph{Richardson orbits for real classical groups},
J.~Algebra, \textbf{286} (2005), no.~2, 361--385. \MR{2128022 (2006g:22012)}


\bibitem[Vog91]{Vogan.1991}
David~A. Vogan, Jr., \emph{Associated varieties and unipotent representations},
  Harmonic analysis on reductive groups ({B}runswick, {ME}, 1989), 315--388, Progr.
  Math., vol. 101, Birkh\"auser Boston, Boston, MA, 1991.
  \MR{1168491 (93k:22012)}

\end{thebibliography}


\providecommand{\bysame}{\leavevmode\hbox to3em{\hrulefill}\thinspace}
\providecommand{\MR}{\relax\ifhmode\unskip\space\fi MR }
\providecommand{\MRhref}[2]{%
  \href{http://www.ams.org/mathscinet-getitem?mr=#1}{#2}
}
\providecommand{\href}[2]{#2}
\renewcommand{\MR}[1]{}

\end{document}